\documentclass[twoside,a4paper,11pt]{article} 
\usepackage{amsfonts, amsbsy, amsmath, amsthm, amssymb, latexsym, enumerate}
\usepackage{mathrsfs}
\usepackage{pdfsync}
\usepackage{caption}
\usepackage{comment}
\usepackage[top=30mm,right=30mm,bottom=30mm,left=30mm]{geometry}
\usepackage{stmaryrd}
\usepackage{bm}
\usepackage{fancyhdr}

\renewcommand{\bf}{\textbf}

 \renewcommand{\to}{\rightarrow}

\newcommand{\R}{\mathbb{R}}
\newcommand{\Z}{\mathbb{Z}}

\newcommand{\characteristic}{\operatorname{char}} 
 
\newcommand{\rad}{\operatorname{rad}}

\newcommand{\im}{\operatorname{Im}} 
\newcommand{\ch}{\operatorname{ch}} 
\newcommand{\m}{\operatorname{m}} 
 
\newcommand{\Lie}{\operatorname{Lie}}
  
\newcommand{\rank}{\operatorname{rank}}

\newcommand{\Hom}{\operatorname{Hom}} 
 
\newcommand{\SL}{\operatorname{SL}}
\newcommand{\Spin}{\operatorname{Spin}}
\newcommand{\Isom}{\operatorname{Isom}}
\newcommand{\SO}{\operatorname{SO}}
\newcommand{\Sp}{\operatorname{Sp}}

\renewcommand\labelenumi{(\roman{enumi})}
\renewcommand\theenumi\labelenumi

\newtheorem{thm}{Theorem}[section] 
\newtheorem{prop}[thm]{Proposition} 
\newtheorem{lem}[thm]{Lemma}

\theoremstyle{definition}
\newtheorem{defn}[thm]{Definition}
\newtheorem{remk}[thm]{Remark} 

\newtheorem{hyp}[thm]{Hypothesis}

\begin{document}

\title{A new family of irreducible subgroups of the orthogonal algebraic groups\footnote{AMS Subject Classification 20G05, 20C33. keywords: algebraic group, irreducible representation, maximal subgroups}}

\author{Mika\"el Cavallin and Donna M. Testerman\footnote{The second author acknowledges the support of the Swiss National Science Foundation grant numbers 200021-153543 and 200021-156583.}}

\date{\today}
\maketitle

\begin{abstract}
Let $n\geq 3,$ and let $Y$ be a simply connected, simple algebraic group of type $D_{n+1}$ over an algebraically closed field $K.$ Also let $X$ be the subgroup of type $B_n$ of $Y,$ embedded in the usual way. In this paper, we correct an error in a proof of a theorem of Seitz \cite[8.7]{seitz}, resulting in the discovery of a new family of triples $(X,Y,V),$ where $V$ denotes a finite-dimensional, irreducible, rational $KY$-module, on which $X$ acts irreducibly. We go on to  investigate the impact of the existence of the  new examples on the classification of the maximal closed connected subgroups of the classical algebraic groups.  
\end{abstract}

\section{Introduction}\label{Introduction}     

Let $K$ be an algebraically closed field of characteristic $p\geq 0.$ In the 1950's, Dynkin determined the maximal closed connected subgroups of the simple classical type linear algebraic groups defined over $K$, assuming $\characteristic(K) = 0$ (see \cite{dynkin, dynkin_exc}); in 1987, 
Seitz \cite{seitz} established an analogous classification in the case where $\characteristic (K)>0$. The main step in both of these classifications is the determination of all triples $(X,Y,V),$ where $Y$ is a simple linear algebraic 
group defined over $K$, $X$ is a proper closed connected subgroup of $Y$, and $V$ is a non-trivial irreducible, finite-dimensional ($p$-restricted if $\characteristic (K)=p>0$) $KY$-module on which $X$ acts irreducibly.  The determination of these so-called ``irreducible triples'' is covered in the work of Dynkin \cite{dynkin, dynkin_exc} (in case $\characteristic (K)=0$), Seitz \cite{seitz} (in case $\characteristic (K) >0$ and $Y$ is a classical group), and Testerman \cite{test} (in case $\characteristic (K) >0$ and $Y$ is of exceptional type). The existence of an irreducible triple of the form $(X,Y,V)$ as above, arising from a rational representation $\rho : Y\to {\rm GL}(V)$, indicates that $\rho(X)$ is \emph{not} maximal in the smallest classical group ${\rm Isom}(V)$ 
containing both $\rho(X)$ and $\rho(Y)$, while  the large majority of tensor-indecomposable irreducible representations of a simple algebraic group give rise to maximal subgroups of the smallest classical group containing the image.

Recently,  the second author's PhD student Nathan Scheinmann discovered an irreducible triple which does not appear in \cite[Theorem 1, Table 1]{seitz}. Namely, take $K$ to be of characteristic 3, and $X = B_3$ embedded in the usual way in $Y=D_4$ as the stabilizer of a non-singular 1-space on the 8-dimensional natural module for $Y$. Consider the irreducible $KY$-module with highest weight $\lambda_1+\lambda_2+\lambda_3$, (here $\lambda_i$, $1\leq i\leq 4$, is a set of fundamental weights for $D_4$ and we label Dynkin diagrams as in \cite{Bourbaki}). The restriction of the highest weight to a maximal torus of $B_3$ shows the existence of a $B_3$-composition factor of highest weight $\omega_1+\omega_2+\omega_3$, ($\omega_i$, $i=1,2,3$, a set of fundamental  weights for $B_3$) and consulting \cite{luebeck}, one sees that these modules are both of dimension $384$, and hence $X$ acts irreducibly on $V$.

The absence of this example from \cite[Table 1]{seitz} is the result of an error in the proof of \cite[8.7]{seitz}. Here, we correct the error in this proof and, in so doing, establish the existence of a whole new family of modules $V$ for the group $Y=D_{n+1}$ on which $X=B_n$ acts irreducibly. For a fixed $n$ and a fixed $p$, there are finitely many modules $V$, but for each $n$ there exist infinitely many primes $p$ for which there is a new example. The precise description of the family is given in Theorem~\ref{Main_result_1} below.

The goal of this paper is two-fold: first we concentrate on the embedding $X=B_n\subset Y=D_{n+1}$ and determine all $p$-restricted irreducible representations of $Y$ whose restriction to $X$ is irreducible, thereby correcting \cite[8.7]{seitz}; see Theorem~\ref{Main_result_1} below. The second goal of the paper is to show that the existence of the new examples has no \emph{further} influence on the main results of \cite{seitz} and \cite{test}. Indeed, the proofs of the main theorems in these two articles depend on an inductive hypothesis, concerning the list of examples for smaller rank groups. The new family of examples for the pair $(B_n,D_{n+1})$ alters the inductive hypothesis and therefore requires one to take these new examples into consideration when  working through all other possible embeddings. This is precisely what has been carried out 
in the proofs of Proposition~\ref{Main_result_2} and Theorem~\ref{Main_result_3}.

\begin{remk} Combining the results \cite[Theorem 1]{seitz}, \cite[Main Theorem]{test}, Theorem~\ref{Main_result_1}, Proposition~\ref{Main_result_2}, and Theorem~\ref{Main_result_3}, we conclude that there are no new examples to be added to \cite[Table 1]{seitz} for pairs $(X,Y)$ different from $(B_n,D_{n+1})$, while Theorem~\ref{Main_result_1} covers the embedding $B_n\subset D_{n+1}$. This assertion is dependent upon our Hypothesis~\ref{inductive_hypothesis} below, where we state explicitly which results from \cite{seitz} are assumed  for the proofs of our main results.
\end{remk}

\subsection*{Statement of results}      

Let $Y=\mbox{Spin}_{2n+2}(K)$ $(n\geq 2)$ be a simply connected, simple algebraic group of type $D_{n+1}$  over $K.$ Also let $X$ be the subgroup of type $B_n,$ embedded in $Y$ in the usual way, as the stabilizer of a non-singular $1$-dimensional subspace of the natural module for $Y.$ Fix $T_Y$ a maximal torus of $Y$ and $B_Y\subset Y$ a Borel subgroup containing $T_Y$. Denote by $\{\lambda_1,\ldots,\lambda_{n+1}\}$ the corresponding set of fundamental weights for $T_Y,$ ordered as in \cite{Bourbaki}, where the natural $(2n+2)$-dimensional $KY$-module has highest weight $\lambda_1.$ 
Let $\sigma$ be a graph automorphism of $Y$ stabilizing $T_Y$, and with $X=Y^\sigma$, the group of $\sigma$-fixed points. Our first main result is the following; the proof is given  in Section \ref{Proof_of_Theorem_1}.

\begin{thm}\label{Main_result_1}
Let $Y=\Spin_{2n+2}(K)$ be a simply connected, simple algebraic group of type $D_{n+1}$ over $K$, $n\geq 2$, and let $X$ be the subgroup of type $B_n,$ as above. Consider a non-trivial, irreducible $KY$-module $V$ having $p$-restricted highest weight $\lambda=\sum_{r=1}^{n+1}{a_r\lambda_r}.$ Then $X$ acts irreducibly on $V$ if and only if $\lambda$ or $\sigma(\lambda)$ is equal to  $\sum_{r=1}^n{a_r\lambda_r},$ with $a_n\neq 0,$ and for all $1\leq i<j\leq n$ 
such that $a_ia_j\neq 0$ and $a_r=0$ for all  $i<r<j,$ we have $p\mid (a_i+a_j +j-i).$
\end{thm}

The set of weights which is listed in \cite[Table 1]{seitz} for the pair $(B_n,D_{n+1})$ is  
$$
\{c\lambda_n, a\lambda_k+b\lambda_n\ :\ abc\ne 0,\  p\mid (a+b+n-k)\}.
$$ 
So we see that the new examples are a generalization of those found by Seitz, 
where one congruence condition is replaced by a set of congruence conditions. (Note that there are new examples only if $n\geq 3$.) It is 
perhaps informative to point out precisely what error occurs in the proof of \cite[8.7]{seitz}, where the embedding $B_{n}\subset D_{n+1}$ is considered. In the proof, Seitz defines a certain vector in the irreducible $KY$-module of highest weight $\lambda$ and 
shows that this vector  is annihilated by all simple root vectors in the Lie algebra of $X$, which then implies that $X$ does not act 
irreducibly.  However if $\lambda$ satisfies the congruence conditions, the vector is in fact the zero vector in $V$ and so does not give rise to a second composition factor as claimed. 
It is also natural to ask how one might discover the given set of congruence conditions, and here we must give credit to the work of Ford in \cite{Ford1}, where he studied irreducible triples of the form $(H,G,V)$, $G$ a simple classical type algebraic group over $K$, $H$ a disconnected closed subgroup of $G$ with $H^\circ$ simple, and $V$ an irreducible $KG$-module on which $H$ acts irreducibly. He discovered a family of irreducible triples for the embedding $D_n.2\subset B_n$, where the highest weight of the irreducible $KB_n$-module satisfies similar congruence conditions. His methods were later applied by Cavallin in \cite{Cavallin_thesis} when studying irreducible $KB_n$-modules having precisely two $D_n$-composition factors.

The second goal of this article is to show that the existence of the new examples for the pair $(X,Y) = (B_n,D_{n+1})$, described by Theorem \ref{Main_result_1}, has no \emph{further} influence on the main theorems in \cite{seitz,test}. To explain the issue which must be addressed and our approach to the problem, we must describe to some extent the strategy of the proof of \cite[Theorem 1]{seitz}. 
First note that the assumption that $X$ acts irreducibly on some $KY$-module implies that $X$ is semisimple. One of the main  techniques used  to determine the triples $(X,Y,V)$ as above involves arguing inductively, working with a suitable embedding $P_X=Q_XL_X\subset P_Y=Q_YL_Y$ of parabolic subgroups, where $Q_X = R_u(P_X)\subset R_u(P_Y) = Q_Y$. Indeed, \cite[2.1]{seitz} implies that if $X$ acts irreducibly on $V,$ then the derived subgroup $L_X'$ acts irreducibly on the commutator quotient $V/[V,Q_Y]$, an irreducible $KL_Y'$-module.  Moreover, the highest weight of $V/[V,Q_Y]$ as $KL_Y'$-module is the restriction of the highest weight of $V$ to an appropriate maximal torus of $L_Y'$. (This is a variation of a result of \cite{smith}.) Thus, Seitz and Testerman proceed by induction on the rank of $X$; Seitz treats the case $X$ of type $A_1$ by ad hoc methods, exploiting the fact that all weights of an irreducible $KA_1$-module are of multiplicity one. Now Theorem \ref{Main_result_1} above introduces a new family of examples of irreducible triples. As a consequence, one needs to reinvestigate all embeddings  $X\subset Y$ where the pair $B_m\subset D_{m+1}$, $m\geq 3$, may arise when considering the projection of a Levi factor $L_X'$ of $X$ into a simple component of a Levi factor $L_Y'$ of $Y$, under the additional hypothesis that $X$ acts irreducibly on a $KY$-module whose highest weight has restriction to the 
$D_{m+1}$-component of $L_Y'$ among the new examples described by Theorem \ref{Main_result_1}. This is precisely what we consider in Proposition \ref{Main_result_2} below. In order to state the  result, we introduce the following terminology.

\begin{defn}\label{Definition:Congruence_conditions}
We will say a $p$-restricted dominant weight $\lambda$ for $Y=D_{m+1}$, $m\geq 3$, \emph{satisfies the congruence conditions} if 
\begin{enumerate}
\item $\lambda=\sum_{i=1}^{m}a_i\lambda_i$, 
\item $ a_m\ne 0$, 
\item there exists $i< j\leq m-1$ with $a_ia_j\ne 0$, and 
\item for all $1\leq i<j\leq m$ such that $a_ia_j\ne 0$ and $a_r=0$ for all $i<r<j$, we have $p \mid (a_i+a_j+j-i)$.
\end{enumerate}
\end{defn}

Note that the above congruence conditions are precisely those satisfied by the highest weights in Theorem~\ref{Main_result_1} but \emph{not} appearing in \cite[Table 1]{seitz}. (See the remark following the statement of Theorem~\ref{Main_result_1}.)

For the proofs of Proposition~\ref{Main_result_2} and Theorem~\ref{Main_result_3}, we require the following inductive hypotheses.

\begin{hyp}\label{inductive_hypothesis} 
Assume $\characteristic (K) = p>0$. Let $G$ be a simple algebraic group defined over $K$ and $H$ a semisimple, proper, closed subgroup of $G$, where the pair $(H,G)$ is one of the following:\
\begin{enumerate}[(i)]
\item $(H,B_n)$, $n\geq 3$,
\item $(B_\ell,A_n)$, $\ell\geq 2$,
\item $(C_3,D_n)$, $n\geq 4$,
\item $(C_3,A_5)$.
\end{enumerate}

 \noindent Let $V$ be a $p$-restricted irreducible $KG$-module, with corresponding representation $\rho:G\to \mbox{GL}(V).$ Then $H$ acts irreducibly on $V$ if and only if the triple $(\rho(H),\rho(G),V)$ appears in 
$\cite[{\rm Table\ 1}]{seitz}$, where the highest weight of $V$ is given up to graph automorphisms of $G$.
\end{hyp}

\subsection*{The classical case}

Let $Y$ be of classical type. The next result ensures that, under the assumption of Hypothesis~\ref{inductive_hypothesis} for all embeddings $H\subset G$ with $\rank(G)<\rank(Y),$ the only new examples of irreducible triples $(X,Y,V)$ are those listed in Theorem~\ref{Main_result_1}.

\begin{prop}\label{Main_result_2} Let $Y$ be a simply connected, simple algebraic group of type $D_{n+1}$, 
$n\geq 4$, $X$ a semisimple, proper, closed subgroup of $Y$ acting irreducibly on a $p$-restricted irreducible $KY$-module $V$ of highest weight $\lambda$, $V$ not the natural module for $Y$. Assume Hypothesis~$\ref{inductive_hypothesis}$ for all embeddings $H\subset G$ with $\rank(G)<\rank(Y).$ Moreover, if $X$ is simple, assume the following conditions are satisfied.  \begin{enumerate}[]
\item{\rm (i)} $P_X$ is a maximal proper parabolic subgroup of $X$, with Levi factor $L_X$ of type $B_{m-1}$, $m\geq 4$. 
\item{\rm (ii)} $P_Y$ is a parabolic subgroup of $Y$ with $P_X\subset P_Y$ and $R_u(P_X)\subset R_u(P_Y)$. 
\item{\rm (iii)} For a Levi factor 
$L_Y$ of $P_Y$, writing $L_Y'=L_1L_2\cdots L_t$, a commuting product of simple groups, $L_t$ is of type 
$D_{m},$ and $L_X'$ projects non-trivially into $L_t.$ 
\item{\rm (iv)} For the irreducible $KL_Y'$-module 
$V/[V,Q_Y]$, write $V/[V,Q_Y]= M_1\otimes \cdots\otimes M_t$, where $M_i$ is an irreducible
 $KL_i$-module. 
\item{\rm (v)} The highest weight of the $KL_t$-module $M_t$ satisfies the congruence 
conditions. 
\end{enumerate}
Then one of the following holds.
\begin{enumerate}[]
\item{\rm (a)} $X\subset B_k\times B_{n-k}$ for some $0<k<n,$ and $\lambda=\lambda_n$ or $\lambda_{n+1}$, or
\item{\rm (b)} $X=B_{n}$ and the embedding of $X$ in $Y$ is the usual embedding of $B_n$ in $D_{n+1}$, that is, $B_n$ is the stablizer of a nonsingular 1-space on the natural $(2n+2)$-dimensional $KY$-module. 
\end{enumerate}
Moreover, in case {\rm (a)} above, the subgroup $B_k\times B_{n-k}$ acts irreducibly on the $KY$-modules of highest weights $\lambda_n$ and $\lambda_{n+1}$.
\end{prop}

\begin{remk} Now to go on to determine the irreducible triples satisfying (a), we rely on Hypothesis~\ref{inductive_hypothesis}(i), and for those satisfying (b), we apply Theorem~\ref{Main_result_1}.

\end{remk}

\subsection*{The exceptional case} We now turn to the consideration of the case where $Y$ is a simply connected, simple algebraic group of exceptional type over $K$ and $X$ is a proper closed, 
connected subgroup of $Y$ acting irreducibly on some $p$-restricted irreducible $KY$-module. As usual, $X$ is then semisimple, and once again, we must consider the possibility of a  parabolic embedding $P_X\subset P_Y,$ with Levi factor $L_X$ of $P_X,$ of type $B_m$, Levi factor $L_Y$ of $P_Y,$ having a simple factor of type $D_{m+1}$, with the action on the commutator quotient arising from a weight which satisfies the congruence conditions. In particular, $Y$ is of type $E_n$, for $n=6,7$ or $8$.

\begin{thm}\label{Main_result_3}
Let $Y$ be a simply connected simple algebraic group of type $E_n$, $6\leq n\leq 8$, defined over $K$ and let $X$ be a semisimple, proper, closed, connected subgroup of $Y,$ having   a proper parabolic subgroup with Levi factor of type $B_m$, for some $m\geq 3$. Assume Hypothesis~$\ref{inductive_hypothesis}$ for all embeddings $H\subset G$ with $\rank(G)<\rank(Y).$ Let $V$ be a non-trivial irreducible $KY$-module with $p$-restricted highest weight $\lambda$.
 Then $X$ acts irreducibly on $V$ if and only 
if $Y=E_6,$ $X=F_4,$ and one of the following holds.
\begin{enumerate}[]
\item{\rm (i)} $\lambda=(p-3)\lambda_1$ or $\lambda=(p-3)\lambda_6$, with $p>3.$
\item{\rm (ii)} $\lambda=\lambda_1+(p-2)\lambda_3$ or $\lambda=(p-2)\lambda_5+\lambda_6$,  with $p>2.$ 
\end{enumerate}
\end{thm}

Note that the existence of the examples arising in Theorem  \ref{Main_result_3} had already been established by Testerman \cite[Main Theorem]{test}. The proof of the only if direction requires us to treat, eventually ruling out, several new potential configurations that arise from Theorem \ref{Main_result_1} in the inductive process, as explained in Section \ref{exceptional}. 

\subsection*{About the proofs}

We conclude this section with a brief discussion of the methods and further remarks on our inductive assumption (Hypothesis~\ref{inductive_hypothesis}). In order to prove Theorem \ref{Main_result_1}, we first show that it is enough to work with the Lie algebras of $Y$ and $X.$ Indeed, as $\lambda$ is $p$-restricted, the irreducible $KY$-module $V$ is generated by a maximal vector $v^+$ for $B_Y$ as a module for the universal enveloping algebra $\mathfrak{U}_Y$ of $\Lie(Y).$ Therefore in order to show that $V|_X$ is irreducible, it suffices to show that $\mathfrak{U}_Yv^+ = \mathfrak{U}_Xv^+$, where $\mathfrak{U}_X$ is the universal enveloping algebra of $\Lie(X)$. We rely on the fact that any irreducible module for $Y$ is self-dual as a $KX$-module (see~\ref{Irreducible_KY-modules_are_self-dual_for_X} below), and apply the techniques developed by Ford in \cite{Ford1}, further investigated by Cavallin in  \cite{Cavallin_thesis}, to establish this generation result.

For the proof of Proposition~\ref{Main_result_2}, we carry out an analysis used by Seitz in \cite[Section 8]{seitz}, but applied specifically to the group $Y=D_{m+1}$. He first shows that a proper closed connected subgroup $X$ acts  irreducibly on a non-trivial irreducible $KY$-module only if either $X$ acts irreducibly and tensor indecomposably on the natural module for $Y$, or the triple $(X,Y,V)$ is known. This part of our proof is not at all original, but we include it for completeness. At this point, however, our proof proceeds along different lines; we compare the commutator series for two different parabolic embeddings and obtain conditions on the highest weight which are compatible with the given congruence conditions only if the pair $(X,Y)$ is $(B_m,D_{m+1})$, which is  handled by Theorem~\ref{Main_result_1}.

For the proof of Theorem~\ref{Main_result_3}, we proceed differently than in \cite{test}; we use the classification of the maximal closed positive-dimensional subgroups of the exceptional type algebraic groups, given in \cite{liebeck-seitz},
which was not available when \cite{test} was written. Hence, we first consider the case where $X$ is maximal, find only the two examples of the theorem and conclude using the main result of \cite{test} for the group $Y=F_4$.

In addition to Hypothesis~\ref{inductive_hypothesis}, we rely upon two further results in \cite{seitz}, namely \cite[Theorem 4.1]{seitz} and \cite[6.1]{seitz}. The first result classifies the irreducible triples $(X,Y,V)$ when ${\rm rank}(X)={\rm rank}(Y)$, the second covers the case where ${\rm rank}(X)=1$. The proofs of these results are completely independent of the results in \cite[Section 8]{seitz}. Finally, we will use the results of \cite[Section 2]{seitz} concerning parabolic embeddings and commutator series in irreducible modules for semisimple groups.

\section{Preliminaries}     

In this section, we introduce the notation that shall be used in the remainder of the paper, and recall some basic properties of rational modules for simple linear algebraic groups. We rely on the standard reference  \cite{Jantzen} for a treatment of this general theory.

\subsection{Notation}     

Let $K$ be an algebraically closed field of characteristic $p \geq 0,$ and let $G$ be a simply connected, simple 
linear algebraic group over $K.$ (All algebraic groups considered here will be linear algebraic groups, even if we
 omit to say so explicitly.) Also fix a Borel subgroup $B=UT$ of $G,$ where $T$ is a maximal torus of $G$ and $U$ 
denotes the unipotent radical of $B.$ Let $\rank(G)=\ell$ and let $\Pi =\{\alpha_1,\ldots, \alpha_\ell\}$ be 
the corresponding 
base of the root system $\Phi=\Phi^+\sqcup \Phi^-$ of $G,$ where $\Phi^+$ and $\Phi^-$ denote the sets of positive and 
negative roots, respectively. Throughout we use the ordering of simple roots as in \cite{Bourbaki}. Let 
$\mathscr{W}$ be 
the 
\emph{Weyl group} of $G,$ and for $\alpha \in \Phi,$ denote by $s_{\alpha}$ the corresponding reflection. 
In addition, let
$$
X(T)=\Hom(T,K^*)
$$
denote the character group of $T$ and write $(-,-)$ for the usual $\mathscr{W}$-invariant  inner product on the space 
$X(T)_{\R}=X(T)\otimes \R.$ Also let $\lambda_1,\ldots,\lambda_\ell$ be the fundamental dominant weights for $T$ corresponding to our choice of base $\Pi,$ that is, $\langle \lambda_i,\alpha_j\rangle =\delta_{ij}$ for $1\leq i\leq j\leq \ell,$ where 
$$
\langle \lambda, \alpha \rangle =\frac{2(\lambda,\alpha)}{(\alpha,\alpha)}
$$
for $\lambda,\alpha\in X(T)$, $\alpha\ne 0$. Set $X^+(T)=\{\lambda\in X(T):\langle \lambda,\alpha\rangle \geq 0 \mbox{ for all } \alpha\in \Pi\}$ and call a character $\lambda\in X^+(T)$ a \emph{dominant $T$-weight} (or simply \emph{dominant weight}, 
if the choice of torus is clear in the context). Finally, we say that $\mu\in X(T)$ is \emph{under} $\lambda\in X(T)$ 
(and we write $\mu\preccurlyeq \lambda$) if $\lambda-\mu=\sum_{r=1}^\ell{c_r\alpha_r}$ for some 
$c_1,\ldots,c_\ell\in \mathbb{Z}_{\geq 0}.$ We also write $\mu \prec \lambda$ to indicate that $\mu\preccurlyeq \lambda$ and $\mu\neq \lambda.$

\subsection{Rational modules}\label{Rational_modules}     

In this section, we recall some elementary facts on weights and multiplicities, as well as basic properties of Weyl and irreducible modules for $G.$  Let $V$ be a finite-dimensional, rational $KG$-module. Then 
$$
V=\bigoplus_{\mu\in X(T)}V_\mu,
$$
where for  $\mu\in X(T),$ we have $V_\mu=\{v\in V:tv=\mu(t)v \mbox{ for all } t\in T\}.$ A weight $\mu\in X(T)$ is called a \emph{weight of $V$} if $V_\mu\neq 0,$ in which case $V_\mu$ is said to be its corresponding \emph{weight space}. Also, we denote by $\m_V(\mu)$ the \emph{multiplicity} of $\mu$ in $V,$ and let $\Lambda(V)=\{\mu\in X(T):V_\mu \neq 0\}$ denote the set of weights of $V$ and write $\Lambda^+(V)=\Lambda(V)\cap X^+(T)$ for the set of dominant weights of $V.$ It is well-known that each weight of $V$ is $\mathscr{W}$-conjugate to a unique dominant weight in $\Lambda^+(V).$ Also, if $\lambda\in X^+(T),$ then $w\lambda\preccurlyeq \lambda$ for every $w\in \mathscr{W}$, and all weights in a $\mathscr{W}$-orbit have the same multiplicity.

A non-zero vector $v^+\in V$ is called a \emph{maximal vector} of weight $\lambda\in \Lambda(V)$ for the pair $(B,T)$  if $v^+\in V_\lambda$ and  $Bv^+\subseteq \langle v^+\rangle _K.$ Now for $\lambda\in X^+(T)$ a dominant weight, we write $V_G(\lambda)$ for the Weyl module having highest weight $\lambda,$ and denote by $L_G(\lambda)$ the unique irreducible quotient of $V_G(\lambda).$ In other words, 
$$
L_G(\lambda)= V_G(\lambda)/\rad (\lambda),
$$
where $\rad (\lambda)$ is the unique maximal submodule of $V_G(\lambda),$ called the \emph{radical} of $V_G(\lambda).$ 
We write $\Lambda(\lambda)$ for $\Lambda(V_G(\lambda))$ and $\Lambda^+(\lambda)$ for $\Lambda^+(V_G(\lambda)).$ Also, 
we denote by $H^0(\lambda)$ the induced $KG$-module having highest weight $\lambda.$ Recall that $H^0(\lambda)$ has a 
unique simple submodule, isomorphic to $L_G(\lambda),$ and that 
$$
\Lambda(H^0(\lambda))=\Lambda(-w_0\lambda),
$$
where $w_0$ denotes the longest element in $\mathscr{W}.$  For $\mu\in X^+(T),$ we write $[V,L_G(\mu)]$ to denote the 
number of times the irreducible $KG$-module $L_G(\mu)$ appears as a composition factor of $V.$ We also use the notation 
$$
U_{\alpha}=\{x_\alpha(c):c\in K\}
$$
to denote the $T$-root subgroup of $G$ corresponding to the root $\alpha\in \Phi$ (that is, $x_\alpha:K\to G$ is a morphism of algebraic groups inducing an isomorphism onto $\im(x_\alpha)$, such that $tx_\alpha(c) t^{-1}=x_\alpha(\alpha(t)c)$ for $t\in T$ and $c\in K$). Finally, we fix a Chevalley basis $\mathscr{B}=\{f_{\alpha},h_r,e_{\alpha} : \alpha \in \Phi^+, 1\leq r\leq \ell\}
$ for the Lie algebra $\Lie(G)$ of $G,$ compatible with our choice of $T\subset B,$ where $e_{\alpha}\in \Lie(G)_{\alpha},$ $f_{\alpha}\in \Lie(G)_{-\alpha}$ are root vectors for $\alpha\in \Phi^+$ and $h_r=[e_{\alpha_r},f_{\alpha_r}]$ for $1\leq r\leq \ell.$ The proof of the following result can be deduced from applying the Poincar\'e-Birkhoff-Witt Theorem \cite{Bourbaki_2} to \cite[A. 6.4]{Curtis}.

\begin{lem}\label{Generators_for_weight_spaces_in_irreducibles}
Let $\lambda\in X^+(T)$ be a $p$-restricted dominant weight for $T,$ and let $V=L_G(\lambda)$. Also let  $v^+\in V$ be a maximal vector of weight $\lambda$ for $B,$ and let $\mu \in \Lambda(V)$. Then for any fixed ordering $\leq$ on $\Phi^+,$ we have  
$$
V_\mu=\left\langle f_{\gamma_1}\cdots f_{\gamma_k}v^+ : k\in \mathbb{Z}_{\geq  0},~\gamma_1,\ldots,\gamma_k\in \Phi^+,~\gamma_1\leq \ldots \leq \gamma_k,~\sum_{r=1}^k{\gamma_r}=\lambda-\mu\right\rangle_K.
$$
\end{lem}

We conclude this section by illustrating how Lemma \ref{Generators_for_weight_spaces_in_irreducibles} can provide information on  weight multiplicities in certain irreducible $KG$-modules in the case where $G$ is of type $A_\ell$ $(\ell \geq 2)$ over $K.$  Consider the dominant $T$-weights $\lambda=a\lambda_1+b\lambda_\ell$ $(a,b\geq 1)$ and $\mu =(a-1)\lambda_1+(b-1)\lambda_\ell.$ Writing $V=L_G(\lambda),$ an application of Lemma \ref{Generators_for_weight_spaces_in_irreducibles} then shows that $V_\mu$ is spanned by $f_{\alpha_1+\cdots+\alpha_r}f_{\alpha_{r+1}+\cdots + \alpha_\ell}v^+,$ for $1\leq r\leq \ell-1,$ together with $f_{\alpha_1+\cdots+\alpha_\ell}v^+,$ where $v^+$ is a maximal vector in $V$ for $B.$ (We used the fact that $f_{\alpha_i+\dots+\alpha_j}v^+=0$ for $1< i\leq j < \ell$ together with the commutator formula.) Finally, we  set 
\begin{equation}
V_{r,\ell} = \langle f_{\alpha_1+\cdots+\alpha_r}f_{\alpha_{r+1}+\cdots +\alpha_\ell} v^+ : 1\leq r\leq \ell-1\rangle_K.
\label{V_mu_for_type_A}
\end{equation}

\begin{prop}\label{8.6_Seitz}
Assume $G$ is of type $A_\ell$ over $K$ for some $\ell \in \Z_{\geq 2}$ and consider the dominant $T$-weight $\lambda=a\lambda_1+b\lambda_\ell\in X^+(T),$ $1\leq a,b<p.$ Also let $\mu=(a-1)\lambda_1+(b-1)\lambda_\ell.$ Then the following assertions are equivalent.
\begin{enumerate}
\item \label{type_A_1} The weight $\mu$ affords the highest weight of a composition factor of $V_G(\lambda).$
\item \label{type_A_2} The inequality $\m_{V_G(\lambda)}(\mu) > \m_V(\mu) $ is satisfied.
\item \label{type_A_3} The generators in \eqref{V_mu_for_type_A} are linearly dependent.
\item \label{type_A_4} The element $f_{\alpha_1+\cdots + \alpha_\ell}v^+$ belongs to $V_{1,\ell}.$
\item \label{type_A_5} The divisibility condition $p\mid a+b+\ell -1$ is satisfied.
\end{enumerate}
\end{prop}

\begin{proof}
Clearly \ref{type_A_1} implies \ref{type_A_2}. Conversely, the only dominant weights $\nu\in \Lambda^+(V)$ such that $\mu \prec \nu \prec \lambda$ are $\lambda-\alpha_1$ (if $a>1$) and  $\lambda-\alpha_\ell$ (if $b>1$). These weights have multiplicity $1$ in $V_G(\lambda)$ and hence none of them can afford the highest weight of a  composition factor of $V_G(\lambda)$ by \cite{premet}. Consequently \ref{type_A_2} implies \ref{type_A_1}. Now an application of Freudenthal's formula yields $\m_{V_G(\lambda)}(\mu)=\ell, $ thus showing that $\ell = \m_{V_G(\lambda)}(\mu) > \m_V(\mu) $ if and only if the $\ell$ generators in \eqref{V_mu_for_type_A} are linearly dependent. Therefore \ref{type_A_2} and \ref{type_A_3} are equivalent as well. Finally, let  $ (\eta_r)_{1\leq r\leq \ell} \in K^\ell$ and set 
$$
w^+= \sum_{r=1}^{\ell-1}{\eta_r f_{\alpha_1+\cdots \alpha_r}f_{\alpha_{r+1}+\cdots +\alpha_\ell}v^+} + \eta_\ell f_{\alpha_1+\cdots+\alpha_\ell}v^+.
$$
A straightforward calculation shows that  $Bw^+=\langle w^+\rangle_K$ if and only if  $p \mid a + b + \ell -1$ and $\eta_\ell\neq 0,$ in which case $w^+=0$ (since $V$ is irreducible and $w^+ \notin V_\lambda).$ This shows  that \ref{type_A_3}, \ref{type_A_4} and \ref{type_A_5} are equivalent, thus completing the proof.
\end{proof}

\section{Proof of Theorem \ref{Main_result_1}}     
\label{Proof_of_Theorem_1}                         

Let $K$ be an algebraically closed field having characteristic $p\geq 0$ and let $Y=\Spin_{2n+2}(K)$ be a simply 
connected, simple algebraic group of type $D_{n+1}$ over $K,$ with $n\geq 2.$ 
Let $X\subset Y$ be the subgroup of type $B_n,$ embedded in the usual way, as the stabilizer of a 
non-singular $1$-dimensional subspace of the natural $(2n+2)$-dimensional module for $Y.$ Fix $T_Y$ a 
maximal torus of $Y$ and $T_X$ a maximal torus of $X$ such that $T_X\subset T_Y$ and let 
$T_Y \subset B_Y,$ $T_X \subset B_X$ denote Borel subgroups of $Y,X$ respectively, with 
$B_X\subset B_Y.$ Let $\Pi(Y)=\{\alpha_1,\ldots,\alpha_{n+1}\}$ be the corresponding base for the root system of $Y,$ and denote by $\{\lambda_1,\ldots,\lambda_{n+1}\}$ the corresponding set of fundamental weights for $T_Y,$ where the natural $KY$-module has highest weight $\lambda_1.$ Let $\sigma$ be a graph automorphism of $Y$ stabilizing $T_Y,$ and with $X=Y^\sigma$, the group of $\sigma$-fixed points. Finally, let $\Pi(X)=\{\beta_1,\ldots,\beta_n\}$ be the base for the corresponding root system of $X,$ associated with the choice of Borel subgroup $B_X$, and denote by $\{\omega_1,\ldots,\omega_n\}$ the associated set of fundamental dominant weights for $T_X.$

\subsection{Preliminary considerations}

For $\sigma$ as above and for a $KY$-module $V,$ let $^\sigma V$ denote the vector space $V$ equipped with the $Y$-action $gv=\sigma(g)v,$ for $g\in Y,$ $v\in V.$ Clearly $^\sigma V$ is irreducible if and only if $V$ is.

\begin{lem}\label{Irreducible_KY-modules_are_self-dual_for_X}
Let $V$ be an irreducible, finite-dimensional, rational $KY$-module. Then   $V|_X$ is self-dual.
\end{lem}

\begin{proof}
Let $\lambda\in X^+(T_Y)$ be the highest weight of $V.$ Then  $V^*$ has highest weight $-w_0 \lambda.$ If $n+1$ is even, then $-w_0=1$ by \cite[Exercise 78]{St1} and so $V$ is self-dual as a $KY$-module, from which the desired result follows. If on the other hand $n+1$ is odd, then $ ^\sigma V \cong V^*$ 	by \cite[Lemma 78]{St1}, yielding $(V |_X)^* =(V^*)|_X \cong  (^\sigma V)|_X  = V|_X,$ since $\sigma(x) =x$ for all $\in X$. The result follows.
\end{proof}

Let $\mathscr{B}_Y=\{f_{\alpha},h_r,e_{\alpha} : \alpha \in \Phi^+(Y), 1\leq r\leq n+1\}$ be  a Chevalley basis  for the Lie algebra $\Lie(Y)$ of $Y,$ compatible with our choice of $T_Y\subset B_Y,$ as in Section \ref{Rational_modules}. As in \cite[Section 8]{seitz}, we may assume that the $B_n$-type subalgebra $\Lie(X)$ of $\Lie(Y)$ is generated by the root vectors
\begin{align}
e_{\beta_r}		&=		e_{\alpha_r} \mbox{ for $1\leq r\leq n-1,$ }	\cr
e_{\beta_n}		&=		e_{\alpha_n} +e_{\alpha_{n+1}},					\cr	
f_{\beta_r}		&=		f_{\alpha_r} \mbox{ for $1\leq r\leq n-1,$ }	\cr
f_{\beta_n}		&=		f_{\alpha_n}+f_{\alpha_{n+1}}.	
\label{generators_for_B}
\end{align}

In particular, we get that $\alpha_j|_{T_X}=\beta_j$ for $1\leq j \leq n-1,$ while $\alpha_{n+1}|_{T_X}=\alpha_n|_{T_X}=\beta_n,$ so that $ \lambda_j|_{T_X}=\omega_j$ for $1\leq j\leq n-1,$ $\lambda_n|_{T_X}=\lambda_{n+1}|_{T_X}=\omega_n$ by \cite[Section 13.2, p.69]{Humphreys1}. Also for $1\leq i\leq j \leq n+1,$  write $f_{i,j}=f_{\alpha_i+\cdots+\alpha_j},$ where we set $f_{i,i}=f_{\alpha_i}$ for $1\leq i\leq n+1$  and $f_{n,n+1}=0$  by convention. In a similar fashion, for $1\leq k\leq n,$ we set 
$$
\hat{f}_{k,n+1}=f_{\alpha_k+\cdots+\alpha_{n-1}+\alpha_{n+1}},
$$ 
where again, we adopt the convention $\hat{f}_{n,n+1}=f_{\alpha_{n+1}}.$ Finally, for $1\leq i < j \leq n-1,$ we set 
\[
F_{i,j}=f_{\alpha_i+\cdots+\alpha_{j-1}+2\alpha_j + \cdots + 2\alpha_{n-1}+\alpha_n+\alpha_{n+1}},
\]
where $F_{i,i+1} =f_{\alpha_i+2\alpha_{i+1}+\cdots+2\alpha_{n-1}+\alpha_n+\alpha_{n+1}}$ and $F_{i,n-1}= f_{\alpha_i+\cdots+\alpha_{n-2}+2\alpha_{n-1} + \alpha_n+\alpha_{n+1}}$ for every $1\leq i\leq n-2.$ We will require the following relations in $\Lie(X).$

\begin{lem}\label{All_f_beta's}
Adopting the notation introduced above, we have
\begin{enumerate}
\item \label{All_f_beta's_1} $f_{\beta_r+\cdots+\beta_s} = \pm f_{r,s}$ for $1\leq r\leq s\leq n-1.$
\item \label{All_f_beta's_2} $f_{\beta_r+\cdots+\beta_n}= \pm  ( f_{r,n}  \pm \hat{f}_{r,n+1})$ for $1\leq r\leq n.$
\item \label{All_f_beta's_3}$f_{r,n+1} \in \Lie(X)$  for $1\leq r\leq n-1.$
\item \label{All_f_beta's_4} $F_{r,s+1}\in \Lie(X)$  for $1\leq r\leq s\leq n-2.$
\end{enumerate}
\end{lem}

\begin{proof}
We start by showing \ref{All_f_beta's_1}, arguing by induction on $0\leq s-r \leq n-2.$ If $r=s,$ then the assertion immediately follows from \eqref{generators_for_B}, so we assume $1\leq r<s\leq n-1$ in the remainder of the proof. We then successively get
\begin{align*}
f_{\beta_r+\cdots + \beta_s}		&=	\pm ( f_{\beta_{r+1}+\cdots +\beta_s}f_{\beta_r} -  f_{\beta_r}f_{\beta_{r+1}+\cdots+\beta_s} )\cr
								&=	\pm N_1( f_{r+1,s}f_{\alpha_r} - f_{\alpha_r}f_{r+1,s}) \cr
								&=	\pm N_1( N_2 f_{r,s} + f_{\alpha_r}f_{r+1,s} - f_{\alpha_r}f_{r+1,s} )\cr		
								& =	\pm N_1N_2 f_{r,s}
\end{align*}
for some $ N_1,N_2 \in \{\pm 1\},$ where the second equality follows from \eqref{generators_for_B} and our induction assumption. Therefore \ref{All_f_beta's_1} holds as desired. For the second assertion, we again argue by induction on $0\leq n-r \leq n-1.$ In the case where $r=n,$ then the result  holds by \eqref{generators_for_B}, hence we assume $1\leq r\leq n-1$ in the remainder of the proof. We then successively get 
\begin{align*}
f_{\beta_r+\cdots+\beta_n}	&=	\pm (f_{\beta_{r+1}+\cdots+\beta_n}f_{\beta_r} - f_{\beta_r}   f_{\beta_{r+1}+\cdots+\beta_n})	\cr
							&=	\pm N_1 ( f_{r+1,n}f_{\alpha_r} + N_2 \hat{f}_{r+1,n+1}f_{\alpha_r} - f_{\alpha_r}f_{r+1,n} -  N_2f_{\alpha_r}\hat{f}_{r+1,n+1} )	\cr
							&= 	\pm N_1 ( N_3 f_{r,n}  + N_2 \hat{f}_{r+1,n+1}f_{\alpha_r}  -  N_2f_{\alpha_r}\hat{f}_{r+1,n+1} ) \cr
							&= 	\pm N_1 ( N_3 f_{r,n}  + N_2N_4 \hat{f}_{r,n+1}) \cr
							&= 	\pm N_1N_3 (f_{r,n}  + N_2N_3N_4 \hat{f}_{r,n+1})
\end{align*}
for some $N_1,N_2,N_3,N_4\in \{\pm 1\},$ where again the second equality follows from \eqref{generators_for_B} and our induction assumption. Therefore \ref{All_f_beta's_2} holds as well. Next we show the third assertion, letting $1\leq r\leq n-1$ be fixed, and setting $\mu=\beta_r+\cdots+\beta_{n-1}+2\beta_n.$ The aforementioned root restrictions yield $\Lie(X)_{\mu}\subset \Lie(Y)_{\alpha_r+\cdots+\alpha_{n+1}},$ and since the latter $T_Y$-weight space is $1$-dimensional, we get that $\Lie(X)_\mu$ is at most $1$-dimensional as well. Now as $f_{\beta_r+\cdots+\beta_n+2\beta_{n+1}}\in \Lie(X)_\mu$ and $\Lie(Y)_{\alpha_r+\cdots+\alpha_{n+1}}=\langle f_{r,n+1}\rangle_K,$ we get that $f_{\beta_r+\cdots+\beta_n+2\beta_{n+1}}$ is a non-zero multiple of $f_{r,n+1},$ from which \ref{All_f_beta's_3} follows. Finally, the assertion \ref{All_f_beta's_4} can be dealt with in a similar fashion.
\end{proof}

In the remainder of this section, we let $V=L_Y(\lambda)$ be a non-trivial, irreducible $KY$-module having $p$-restricted highest weight $\lambda=\sum_{r=1}^{n+1}{a_r\lambda_r}\in X^+(T_Y),$ and fix a maximal vector $v^+$ in $V$ for $B_Y.$ Setting $\lambda|_{T_X}=\omega,$ one observes that $v^+$ is a maximal vector of weight $\omega$ in $V$ for $B_X,$ since $B_X\subset B_Y.$ The following result provides a necessary and sufficient condition for $V$ to be irreducible in the case where $a_n\neq 0 =a_{n+1}.$

\begin{lem}\label{V|X_irreducible_iff_V=Lie(X)v+}
Let $V=L_Y(\lambda)$ be as above, and assume $a_n  \neq 0=a_{n+1}.$  Then  $X$  acts irreducibly on $V$ if and only if $V=\Lie(X)v^+.$
\end{lem}
 
\begin{proof}
First assume $X$ acts irreducibly on $V,$ so that $V|_X\cong L_X(\omega),$ and observe that since $a_n a_{n+1}  =0,$ the $T_X$-weight $\omega$ is $p$-restricted. Therefore the Lie algebra $\Lie(X)$ of $X$ acts irreducibly on $V$   by \cite[Theorem 1]{Curtis}, from which the desired assertion follows. Conversely, assume $V=\Lie(X)v^+.$ Then $V=\langle Xv^+\rangle$ and hence $V|_X$ has a quotient isomorphic to $L_X(\omega)$ by \cite[II, Lemma 2.13 (b)]{Jantzen}. Consequently $  V^* |_X$ contains a $KX$-submodule isomorphic to $L_X(\omega).$ Now  $V^*|_X  \cong V|_X$ by Lemma \ref{Irreducible_KY-modules_are_self-dual_for_X}, showing the existence of a submodule $U$ of  $V|_X$ such that $U\cong L_X(\omega).$ Since $(V|_X)_\omega =\langle v^+\rangle_K,$ we get that $v^+\in U$ and so $V=\langle X v^+\rangle  \subset U \cong  L_X(\omega)$ as desired.
\end{proof}

In view of Lemma \ref{V|X_irreducible_iff_V=Lie(X)v+}, a necessary condition for $X$ to act irreducibly on $V=L_Y(\lambda)$, with $\lambda$ such that $a_n\neq 0=a_{n+1}$, is for $f_{\gamma}v^+$ to belong to $\Lie(X)v^+$ for every $\gamma\in \Phi^+(Y).$  We conclude this section by showing that $V|_X$ is irreducible if and only if $f_{r,n}v^+ \in \Lie(X)v^+$ for  every $1\leq r\leq n$ (see Proposition \ref{f_gamma_1...f_gamma_rv^+_in_Lie(X)v^+} below). We first need the following preliminary lemma.

\begin{lem}\label{f_gammaf_delta_1...f_delta_rv^+_in_Lie(X)v^+}
Let $V=L_Y(\lambda)$ be as above, with $a_n\neq 0=a_{n+1},$ and assume  $f_{\gamma}v^+ \in \Lie(X)v^+$ for every $\gamma\in \Phi^+(Y).$ Then $f_{\gamma}f_{\eta_1}\cdots f_{\eta_s}v^+\in \Lie(X)v^+$ for every $\gamma\in \Phi^+(Y)$ and $\eta_1,\ldots,\eta_s \in \Phi^+(X).$
\end{lem}

\begin{proof}
We proceed by induction on $s\geq 1.$ First take $\eta\in\Phi^+(X)$ and consider $f_\gamma f_\eta v^+,$ with $\gamma\in \Phi^+(Y).$ Since $f_\gamma f_\eta v^+ =[f_{\gamma},f_{\eta}]v^+ + f_{\eta}f_{\gamma}v^+$ and as $f_{\gamma}v^+\in \Lie(X)v^+$ by assumption, it suffices to show  that $[f_{\gamma},f_{\eta}]v^+\in \Lie(X)v^+.$  If $f_\gamma\in\Lie(X),$ then clearly $[f_\gamma,f_\eta]v^+$ lies in $\Lie(X)v^+$, so assume $f_\gamma\not\in\Lie(X).$ By Lemma \ref{All_f_beta's}, we then have that 
$$
\gamma \in\{\alpha_n,\alpha_{n+1}\}\cup \{\alpha_i+\cdots+\alpha_n, 
\alpha_i+\cdots+\alpha_{n-1}+\alpha_{n+1}\}_{1\leq i\leq n-1}.
$$

Take first $\gamma= \sum_{r=i}^n{\alpha_r}$ for some $1\leq i\leq n,$ so $f_\gamma = f_{i,n},$ and consider $[f_\gamma ,f_\eta] v^+$. 
If the latter equals zero, then we immediately get the desired 
result. So we may assume the contrary and thus we have $f_\eta$ is one of the following:
\begin{enumerate}
\item \label{computation1} $f_{j,i-1}$, $1\leq j\leq i-1.$
\item \label{computation2} $f_{\alpha_n}+f_{\alpha_{n+1}}$ (if $i\ne n$).
\item \label{computation3} $f_{k,n}\pm\hat f_{k,n+1}$, $ i+1 \leq k\leq n.$ 
\item \label{computation4} $f_{j,n}\pm\hat f_{j,n+1}$, $1\leq j\leq i-1.$
\end{enumerate}
We now calculate $[f_\gamma,f_\eta] v^+$ in each case.
\begin{enumerate}
\item $[f_{i,n},f_{j,i-1}]v^+ = \pm f_{j,n}v^+,$ which lies in $\Lie(X)v^+$ by assumption.
\item $[f_{i,n},f_{\alpha_n}+f_{\alpha_{n+1}}]v^+ =  \pm f_{i,n+1}v^+,$ which  lies in $\Lie(X)v^+$ since $f_{i,n+1}\in \Lie(X)$ by Lemma \ref{All_f_beta's}   \ref{All_f_beta's_3}.
\item  $[f_{i,n},f_{k,n}\pm\hat f_{k,n+1}]v^+ = \pm F_{i,k}v^+,$ which again lies in $\Lie(X)v^+$ as $F_{i,k}\in\Lie(X)$ by Lemma \ref{All_f_beta's}  \ref{All_f_beta's_4}.
\item $[f_{i,n},f_{j,n}\pm \hat f_{j,n+1}]v^+ =  \pm F_{j,i}v^+$, which 
as in the previous case  lies in $\Lie(X)v^+.$
\end{enumerate}
Consequently $f_{\gamma}f_{\eta}v^+\in \Lie(X)v^+$ for $\gamma=\sum_{r=i}^{n}{\alpha_r}$ as desired. Arguing in a similar fashion, one shows that the same holds for $\gamma= \sum_{r=i}^{n-1}{\alpha_r}+\alpha_{n+1}$ and $\gamma =\alpha_{n+1}$ as well, where $1\leq i\leq n-1.$ Therefore the lemma holds in the situation where $s=1,$ and  hence we assume $s>1$ in the remainder of the proof. For $\eta_1,\ldots,\eta_s\in \Phi^+(X),$ we have
$$
f_{\gamma}f_{\eta_1}\cdots f_{\eta_s}v^+=[f_\gamma,f_{\eta_1}]f_{\eta_2}\cdots f_{\eta_s}v^+ + f_{\eta_1}f_{\gamma}f_{\eta_2}\cdots f_{\eta_s}v^+.
$$
Now $f_{\gamma}f_{\eta_2}\cdots f_{\eta_s}v^+\in \Lie(X)v^+$ by induction and hence so does $f_{\eta_1}f_{\gamma}f_{\eta_2}\cdots f_{\eta_s}v^+.$ On the other hand, we either have $[f_{\gamma},f_{\eta_1}] =0$ or $[f_{\gamma},f_{\eta_1}] = \xi f_{\delta}$ for some $\xi \in K$ and $\delta\in \Phi^+(Y)$ by \ref{computation1}, \ref{computation2}, \ref{computation3}, \ref{computation4} above, in which case  $[f_\gamma,f_{\eta_1}]f_{\eta_2}\cdots f_{\eta_s}v^+= \xi f_\delta f_{\eta_2}\cdots f_{\eta_s}v^+\in \Lie(X)v^+$ by induction, thus completing the proof.	
\end{proof}

We now establish the following  necessary and sufficient condition for $V|_{\Lie(X)}$ to be generated by $v^+$ as a $\Lie(X)$-module (and hence for $V|_X$ to be irreducible by Lemma \ref{V|X_irreducible_iff_V=Lie(X)v+}), where $V$ has highest weight $\lambda$ with $a_n\neq 0=a_{n+1}.$

\begin{prop}\label{f_gamma_1...f_gamma_rv^+_in_Lie(X)v^+}
Let  $V=L_Y(\lambda)$ be an irreducible $KY$-module having $p$-restricted highest weight $\lambda=\sum_{r=1}^na_r\lambda_r,$ with $a_n  \neq  0.$  Then  $V=\Lie(X)v^+$ if and only if $f_{r,n}v^+\in \Lie(X)v^+$ for every $1\leq r\leq n.$
\end{prop}

\begin{proof}
If $V=\Lie(X)v^+,$ then clearly $f_{\gamma}v^+\in \Lie(X)v^+$ for every $\gamma\in \Phi^+(Y).$ Hence in particular $f_{r,n}v^+ \in \Lie(X)v^+$ for every $1\leq r\leq n$ as desired. Conversely, suppose that  $f_{r,n}v^+ \in \Lie(X)v^+$ for every $1\leq r\leq n$ and notice that since $\lambda\in X^+(T_Y)$ is $p$-restricted, we have 
\[
V=\left\langle f_{\gamma_1}\cdots f_{\gamma_s} v^+ : s\in \Z_{\geq 0},~\gamma_1,\ldots,\gamma_s\in \Phi^+(Y)\right\rangle_K
\]
by \cite[Theorem 1]{Curtis}. Hence in order to show that $V=\Lie(X)v^+,$ it suffices to show that $f_{\gamma_1}\cdots f_{\gamma_s}v^+\in \Lie(X)v^+$ for every $s\in \Z_{>0}$ and  $\gamma_1,\ldots,\gamma_s\in \Phi^+(Y).$ Assume for a contradiction that this is not the case and let $m\in \mathbb{Z}_{\geq 0}$  be minimal such that there exists $\gamma_1,\ldots,\gamma_m\in \Phi^+(Y)$ with  $f_{\gamma_1}\cdots f_{\gamma_m}v^+\notin \Lie(X)v^+.$ Lemma \ref{All_f_beta's}  implies $f_{\gamma}v^+\in \Lie(X)v^+$ for $\gamma\in \Phi^+(Y)$ and so $m\geq 2.$  Now by minimality, we have  $f_{\gamma_2}\cdots f_{\gamma_m}v^+\in \Lie(X)v^+$ and hence an application of Lemma \ref{f_gammaf_delta_1...f_delta_rv^+_in_Lie(X)v^+} completes the proof.
\end{proof}

\subsection{Conclusion}\label{The_Proof}     

Let $V=L_Y(\lambda)$  be an irreducible $KY$-module having $p$-restricted non-zero highest weight $\lambda=\sum_{r=1}^{n+1}{a_r\lambda_r} \in  X^+(T_Y),$ and set $\omega=\lambda|_{T_X}.$ In this section, we will complete the proof of Theorem \ref{Main_result_1}. We first show that for certain weights $\lambda,$ it is straightforward to see that $V|_X$ is reducible. Although the proof of the following proposition can be found in 
\cite[Section 8]{seitz}, we include it here for completeness.

\begin{prop}\label{A_first_reduction}
Let $\lambda\in X^+(T_Y)$ and $V$ be as above. Then the following assertions hold.
\begin{enumerate}
\item \label{1st_cond_reduction_lemma} If $a_n a_{n+1}  \neq 0,$ or if $a_n   =  a_{n+1} =0,$   then $V|_X$ is reducible. 
\item \label{2nd_cond_reduction_lemma} If  $\lambda\in \{a\lambda_n, ~a\lambda_{n+1}:  a\in \mathbb{Z}_{> 0} \},$ then $X$ acts irreducibly on $V.$
\item  \label{3rd_cond_reduction_lemma} If $V|_X$ is irreducible, so $a_na_{n+1}=0$ by \textnormal{\ref{1st_cond_reduction_lemma}},  and $\lambda$ is not as in \textnormal{\ref{2nd_cond_reduction_lemma}}, then taking $1\leq k <n$ maximal such that $a_k\neq 0,$ we have  $p\mid (a_k+a_n+a_{n+1}+n-k).$ 
\end{enumerate}  
\end{prop}

\begin{proof}
For \ref{1st_cond_reduction_lemma}, first consider the case  where $a_n a_{n+1}  \neq 0.$ Here the $T_X$-weight $\omega'=\omega-\beta_n$ has multiplicity at most $1$ in $L_X(\omega),$ while each of $\lambda-\alpha_n$ and $\lambda-\alpha_{n+1}$ is a $T_Y$-weight of $V$ restricting to $\omega'.$ Therefore the latter occurs in a second $KX$-composition factor of $V,$ showing that $V|_X$ is reducible as desired. Next assume $a_n  = a_{n+1}  =0,$ and let $1\leq k\leq n-1$ be maximal such that $a_k  \neq 0.$ Then the $T_X$-weight $\omega''=\omega-(\beta_k+\cdots+\beta_n)$ has multiplicity at most $1$ in $L_X(\omega)$ (since the corresponding weight space is generated by $f_{\beta_k+\cdots+\beta_n}v^+$), while each of  $\lambda-(\alpha_k+\cdots+\alpha_n)$ and $\lambda-(\alpha_k+\cdots+\alpha_{n-1}+\alpha_{n+1})$ restricts to $\omega''.$ Consequently $V|_X$ is reducible in this case as well, and \ref{1st_cond_reduction_lemma} holds as desired. 

Now turn to \ref{2nd_cond_reduction_lemma} and let $\lambda=a\lambda_n$ for some $a\in \Z_{>0},$ and assume for a contradiction that $V|_X$ is reducible. By \cite[8.5]{seitz}, there  exist $1\leq i\leq n$ and a maximal vector $w^+ \in V$ for $B_X$ such that  
$$
0\neq e_{\alpha_i+\cdots+\alpha_n}w^+ \in \langle v^+ \rangle_K.
$$
In particular $w^+ \in V_{\lambda-(\alpha_i+\cdots + \alpha_n)},$ and one checks that the only $T_Y$-weight of $V$ restricting to $\omega-(\beta_i+\cdots +\beta_n) $ is $\lambda-(\alpha_i+ \cdots + \alpha_n).$ So $w^+ \in V_{\lambda-(\alpha_i+\cdots+\alpha_n)}=\langle f_{\alpha_i+\cdots +\alpha_n}v^+\rangle_K,$ thus yielding $e_{\beta_i}w^+ =e_{\alpha_i}w^+ \neq 0,$ a contradiction. Therefore \ref{2nd_cond_reduction_lemma} holds as desired.

For \ref{3rd_cond_reduction_lemma}, we assume $V|_X$ is irreducible, and by  \ref{1st_cond_reduction_lemma}, we suppose, without loss of generality, that $a_n\neq 0 = a_{n+1}.$  Assume as well that $\lambda$ is not as in \ref{2nd_cond_reduction_lemma}, and seeking a contradiction, take $k$ as in \ref{3rd_cond_reduction_lemma}, and suppose
 that $p\nmid (a_k+a_n +n-k).$ Consider the $T_X$-weight $\omega''=\omega-(\beta_k+\cdots+\beta_n)\in X^+(T_X).$ Then  $\lambda-(\alpha_k+\cdots+\alpha_n)$ and $\lambda-(\alpha_k+\cdots+\alpha_{n-1}+\alpha_{n+1})$ are both $T_Y$-weights of $V$ restricting to $\omega''.$ Now $\m_{L_X(\omega)}(\omega'')\leq \m_{V_X(\omega)}(\omega'') = n-k+1,$ while an application of Proposition \ref{8.6_Seitz} yields  $\m_V(\lambda-(\alpha_k+\cdots+\alpha_n))=n-k+1.$ (Recall that we assumed $a_n\neq 0 = a_{n+1}$ and $p\nmid (a_k + a_n +n-k).$) Now $\m_V(\lambda-(\alpha_k+\cdots+\alpha_{n-1}+\alpha_{n+1}))=1,$ as the latter weight is conjugate to $\lambda-\alpha_k$ under the action of the Weyl group for $Y.$ Therefore $\m_{V|_X}(\omega'') \geq n-k+2 > n-k+1 \geq  \m_{L_X(\omega)}(\omega''),$ yielding the existence of a second composition factor of $V$ for $X,$ a contradiction. The proof is complete.
\end{proof}

In view of Proposition \ref{A_first_reduction}, we may and shall assume $\lambda=\sum_{r=1}^n{a_r \lambda_r},$ with $a_n \neq 0  $ throughout the rest of the section, as well as the existence of $1\leq  k\leq n-1$  maximal such that $  a_k \neq 0.$ For $1\leq i\leq n,$ set $P(i,i)=\emptyset$ and for $1\leq i<j\leq n,$ set 
$$
P(i,j)=\left\{(m_r)_{r=1}^s : 1\leq s \leq j-i,~ i\leq m_1<\ldots < m_s<j\right\}.
$$
For any sequence $(m)=(m_r)_{r=1}^s\in P(i,j),$ write $f_{(m)}=f_{i,m_1}f_{m_1+1,m_2}\cdots f_{m_s+1,j}.$ By Lemma \ref{Generators_for_weight_spaces_in_irreducibles}, we have that for every $1\leq i<j\leq n,$ the weight space $V_{\lambda-(\alpha_i+\cdots+\alpha_j)}$ is spanned by the vectors
\begin{equation*}
\{f_{(m)}v^+ : (m)\in P(i,j)\} \cup \{f_{i,j}v^+\},
\end{equation*}
where $v^+\in V$ is a maximal vector of weight $\lambda$ for $B_Y.$ We set 
$$
V_{i,j}=\left\langle f_{(m)}v^+:(m)\in P(i,j)\right\rangle_K.
$$
The following special case of \cite[Theorem A.7]{Cavallin3}, inspired by \cite[Proposition 3.1]{Ford1}, shall play a key role in the proof of Theorem \ref{Main_result_1}.

\begin{thm}\label{How_to_relate_divisibility_conditions_and_generators_for_V_i,j}
Let $Y$ be a simple algebraic group of type $D_{n+1}$ over $K,$ and consider an irreducible $KY$-module $V=L(\lambda)$ having $p$-restricted highest weight $\lambda=\sum_{r=1}^n{a_r \lambda_r},$ with $a_n\neq 0.$ Then $f_{r,n}v^+\in V_{r,n}$ for every $1\leq r \leq n-1$ if and only if $p\mid (a_i+a_j+j-i)$ for every $1\leq i<j\leq n$ such that $a_ia_j\neq 0$ and $a_s=0$ for $i<s<j.$
\end{thm}

\begin{proof}
The result follows from an application of \cite[Theorem A.7]{Cavallin3} to the $A_n$-Levi subgroup of $Y$ corresponding to the simple roots $\alpha_1,\ldots, \alpha_n.$  
\end{proof}

For $1\leq i\leq n-1,$ we let $\hat{V}_{i,n+1}$ denote the $K$-span of $\{f_{i,j}\hat{f}_{j+1,n+1}v^+: i\leq j\leq n-1\}$ and $\{f_{(m)}\hat{f}_{j+1,n+1}v^+ : i < j\leq n-1,~ (m)\in P(i,j)\}.$ The proof of the main result of this section (namely, Theorem \ref{Key_Theorem}) relies on the following preliminary result.

\begin{lem}\label{Lemma_on_V_r,n_VS_Vr,n+1}
Adopt the notation introduced above and let $1\leq r \leq k-1$ be such that  $\hat{f}_{r,n+1}v^+\in \hat{V}_{r,n+1}.$ Then $f_{r,n}v^+\in V_{r,n}.$  
\end{lem}

\begin{proof}
Let $1\leq r \leq k-1 $ be as in the statement of the lemma. If $a_r =0,$ then $f_{\alpha_r}v^+=0$ and hence $f_{r,n}v^+ = \pm f_{\alpha_r}f_{r+1,n}v^+,$ so that $f_{r,n}v^+ \in V_{r,n}$, as claimed. Therefore we assume $a_r\neq 0$ in the remainder of the proof. As $\hat{f}_{r,n+1}v^+\in \hat{V}_{r,n+1}$ and since $\hat{f}_{s,n+1}v^+=0$ for $k<s\leq n,$ we get the existence of $(\xi_s,\xi^s_{(m)} : r\leq s\leq k-1, (m)\in P(r,s) )\subset K$ such that
\begin{equation}
\hat{f}_{r,n+1}v^+ = \sum_{s=r}^{k-1}{\left(\sum_{(m)\in P(r,s)}{\xi^s_{(m)}f_{(m)}\hat{f}_{s+1,n+1}v^+ } + \xi_{s}f_{r,s}\hat{f}_{s+1,n+1}v^+ \right)}.
\label{equation_in_Lemma_on_V_r,n_VS_Vr,n+1}
\end{equation}
 
Now observe that $f_{\alpha_n} e_{\alpha_{n+1}}\hat{f}_{t,n+1}v^+  =  N_1 f_{t,n}v^+ + N_2 f_{t,n-1}f_{\alpha_n}v^+$ for every $r\leq t\leq k,$ where $N_1,N_2 \in \{\pm 1\}$. Also, for every $r\leq s\leq k-1$ and every $(m)\in P(r,s),$ we have $[e_{\alpha_{n+1}},f_{r,s}]=[e_{\alpha_{n+1}},f_{(m)}]=0,$ as well as $[f_{\alpha_{n}},f_{r,s}]=[f_{\alpha_{n}},f_{(m)}]=0.$ Consequently successively applying $e_{\alpha_{n+1}},$ $f_{\alpha_n} $  to each side of   \eqref{equation_in_Lemma_on_V_r,n_VS_Vr,n+1}   yields $N_1 f_{r,n}v^+ + N_2 f_{r,n-1}f_{\alpha_n}v^+ \in V_{r,n}$ for some $N_1\neq 0,$ $N_2\in \{\pm 1\},$  from which the desired result follows.
\end{proof}

In the next result, we show that in order to determine whether $V|_X$ is irreducible or not, it is enough to determine whether $f_{r,n}v^+ \in V_{r,n}$ or not, this for every $1\leq r\leq n-1.$

\begin{thm}\label{Key_Theorem} 
Let $\lambda=\sum_{r=1}^n{a_r\lambda_r}$ be a $p$-restricted dominant weight such that $a_ka_n\ne 0.$ Let $V=L_Y(\lambda)$ be an irreducible $KY$-module having highest weight $\lambda.$  Then $V|_X$ is irreducible if and only if $f_{r,n}v^+\in V_{r,n}$ for every $1\leq r \leq n-1.$ 
\end{thm}
\begin{proof}
First assume $f_{r,n}v^+\in V_{r,n}$ for every $1\leq r\leq n-1.$ By  Proposition \ref{f_gamma_1...f_gamma_rv^+_in_Lie(X)v^+}, in order to show that $V|_X$ is irreducible, it suffices to show that $f_{s,n}v^+\in \Lie(X)v^+$ for every $1\leq s\leq n.$  We proceed by induction on $0\leq n-s\leq n-1,$ starting by assuming $s=n.$ Now since  $a_{n+1}  =0,$ we immediately get  that $f_{\alpha_{n+1}}v^+=0$ and hence $f_{\beta_n}v^+=f_{\alpha_n}v^+$  by \eqref{generators_for_B}, that is, $f_{\alpha_n}v^+\in \Lie(X)v^+.$ Next  assume $1\leq s\leq n-1.$ Since $f_{s,n}v^+\in V_{s,n}$ by assumption, there exists $(\xi_i, \xi^i_{(m)}:  s\leq i\leq n-1, (m)\in P(s,i))\subset K$ such that
\[
f_{s,n}v^+ =  \sum_{i=s}^{n-1} {\left( \xi_{i}f_{s,i}f_{i+1,n}v^+ +\sum_{(m)\in P(s,i)}{ \xi^i_{(m)}f_{(m)}  f_{i+1,n}v^+}\right)}.
\]
By Lemma \ref{All_f_beta's}, $f_{s,i}\in \Lie(X)$ and $f_{(m)} \in \Lie(X) $ for every $s\leq i \leq n-1$ and   every $(m)\in P(s,i).$ Furthermore, we also have $f_{i+1,n}v^+\in \Lie(X)v^+$ for $s\leq i \leq n-1$ by induction.  Therefore  $f_{s,n}v^+\in \Lie(X)v^+$ as desired. 

Conversely, assume $V|_X$ irreducible, and let $1\leq r\leq n-1$ be fixed. If $a_r =0,$ then $f_{\alpha_r}v^+=0$ and hence $f_{r,n}v^+ = \pm f_{\alpha_r}f_{r+1,n}v^+,$ so that $f_{r,n}v^+ \in V_{r,n}$ in this situation.  In addition, observe that $p\mid (a_k+a_n+n-k)$ by Proposition \ref{A_first_reduction} \ref{3rd_cond_reduction_lemma}, thus yielding $f_{k,n}v^+\in V_{k,n}$ by Proposition  \ref{8.6_Seitz}. Therefore we assume $a_r\neq 0$ in the remainder of the proof, and  $r<k.$ By Lemma \ref{V|X_irreducible_iff_V=Lie(X)v+}, we have $V=\Lie(X)v^+,$ in which case Proposition \ref{f_gamma_1...f_gamma_rv^+_in_Lie(X)v^+} applies, thus yielding $f_{r,n}v^+\in \Lie(X)v^+.$ Since $f_{r,n}v^+\in (\Lie(X)v^+)_{\omega-(\beta_r+\cdots+\beta_n)},$ we get the existence of $\xi \in K,$ $x\in V_{r,n}$ and $y\in \hat{V}_{r,n+1}$ such that  $ f_{r,n}v^+ = \xi f_{r,n}v^+  \pm \xi \hat{f}_{r,n+1}v^+ + x + y.$  Comparing $T_Y$-weights yields 
$$
(\xi-1)f_{r,n}v^+ \in V_{r,n} \mbox{ and } \pm \xi \hat{f}_{r,n+1}v^+ 	 \in \hat{V}_{r,n+1}.
$$ 
If $\xi \neq 1,$ then the assertion is immediate, while if on the other hand $\xi =1,$ then an application of Lemma \ref{Lemma_on_V_r,n_VS_Vr,n+1} yields the desired result.
\end{proof}

We are now able to complete the proof of Theorem \ref{Main_result_1}.

\begin{proof}[\textbf{Proof of Theorem  \ref{Main_result_1}}] First assume $X$ acts irreducibly on $V.$
 By Proposition~\ref{A_first_reduction}\ref{1st_cond_reduction_lemma}, we then have that, up to a 
graph automorphism of $Y$, $a_n\neq 0=a_{n+1}$, and if $\lambda=a\lambda_n$, then $\lambda$ is as in 
the statement of the result. So assume the existence of $1\leq k\leq n-1$ maximal with $a_k\ne 0$. 
Then Theorem \ref{Key_Theorem}  applies, yielding $f_{r,n}v^+\in V_{r,n}$ for every $1\leq r\leq n-1.$ 
An application of Theorem \ref{How_to_relate_divisibility_conditions_and_generators_for_V_i,j} 
then implies the desired divisibility conditions.

Conversely, assume $\lambda\in X^+(T_Y)$ satisfies the  conditions in Theorem \ref{Main_result_1}. In the case where $\lambda=a\lambda_n$ for some $a\in \mathbb{Z}_{>0},$ then the result follows from  Proposition 
~\ref{A_first_reduction}\ref{2nd_cond_reduction_lemma}, hence we assume the existence of $1\leq k\leq n-1$ maximal such that $a_ka_n\neq 0,$ and assume moreover the divisibility conditions as in the theorem. Here an application of Theorem~\ref{How_to_relate_divisibility_conditions_and_generators_for_V_i,j}  yields  $f_{r,n}v^+\in V_{r,n}$ for every $1\leq r \leq n-1,$ and hence Theorem~\ref{Key_Theorem} then shows that $X$ acts irreducibly on $V$ as desired. 
\end{proof}

\section{Proof of Proposition \ref{Main_result_2}}

Let $Y=\Spin_{2n+2}(K)$ be a simply connected, simple algebraic group of type $D_{n+1}$ over $K,$ with $n\geq 2.$  Let $X\subset Y$ be a semisimple, connected, 
proper, closed subgroup of $Y.$ Fix $T_Y$ a maximal torus of $Y$  and let $T_Y \subset B_Y $  denote a  Borel subgroup  of $Y.$  Also let $\{\lambda_1,\ldots,\lambda_{n+1}\}$  be the corresponding set of fundamental dominant weights for $T_Y.$ In this section, we give a proof of Proposition  \ref{Main_result_2}, starting with three results, proven in a more general setting in \cite[Section 5]{seitz}. For the convenience of the reader, and in order to render the current manuscript more self-contained, we include the proofs of these special cases here.

\begin{prop}\label{reduction}
Let $Y$ be a simply connected simple algebraic group of type $D_{n+1}$ with natural module $W$ and let $X$ be a semisimple, connected, proper, closed subgroup of $Y$ acting irreducibly on a non-trivial, $p$-restricted, irreducible $KY$-module $V=L_Y(\lambda)$. Then one of the following holds.
\begin{enumerate} 
\item \label{reduction_1} $W|_X$ is irreducible.
\item \label{reduction_2} $W|X$ is reducible and $X\subset B_k\times B_{n-k}$, for some $0\leq k<n$. 
\end{enumerate}
\end{prop}

\begin{proof} Suppose $W|_X$ is reducible. Let $W_1$ be a minimal non-zero $X$-invariant subspace of $W$. Then $W_1\cap W_1^\perp = W_1$ or $\{0\}$. Suppose  $W_1\cap W_1^\perp = W_1$. Note that $W_1$ is not totally singular, else $X$ lies in a proper parabolic subgroup of $Y$ and so cannot act irreducibly on $V$. Therefore, we have $p=2$; the set of singular vectors in $W_1$ being an $X$-invariant subspace of $W_1$ forces $W_1$ to be generated by a non-singular vector, and so \ref{reduction_2} holds with $k=0$. In case $W_1\cap W_1^\perp=\{0\}$, set $W_2 = W_1^\perp$, so that $W=W_1\oplus W_2$, an orthogonal direct sum and the image of $X$ in $\Isom(W)$ lies in $\Isom(W_1)'\times \Isom(W_2)'$. If $\dim W_1$ is even, then $X\subset D_s\times D_{n+1-s}$, a maximal rank subgroup of $Y$ and we may invoke \cite[Theorem 4.1]{seitz} to see that $X$ acts reducibly on $V$. So $\dim W_1$ is odd and $X\subset B_k\times B_{n-k}$ as in \ref{reduction_2}.
\end{proof}

\begin{prop}\label{semisimple} 
Let $Y$ and $V$ be as in Proposition~\textnormal{\ref{reduction}\ref{reduction_2}} for some $0<k<n$, and let $V$ be of highest weight $\lambda$. Then the closed, connected subgroup $H\subset Y$ of type $B_k\times B_{n-k}$ acts irreducibly on $V$ if and only if $\lambda = \lambda_n$ or $\lambda_{n+1}$.
\end{prop}

\begin{proof} To see that $H$ acts irreducibly on the two half-spin modules for $Y$, we simply compute 
the restriction of the highest weight to a maximal torus of $H,$ note that the restriction induces the tensor product of the spin modules for the two simple factors of $H,$ and then a dimension comparison completes the proof.

We now assume $H$ acts irreducibly on $V$, the irreducible $KY$-module of highest weight $\lambda$. We proceed by induction on $n$ and first consider the case $n=2,$ so that $Y$ is of type $D_3 = A_3$ and the image of $H$ in $\SO_6$ lies in the subgroup $\SO_3\times \SO_3\subset \SO_6$. Since $H$ acts irreducibly on $V$, $H$ does not lie in a parabolic subgroup of $Y$ and so acts irreducibly on the $4$-dimensional $p$-restricted $KY$-modules, which correspond to the highest weights as in the statement of the result. Thus the preimage of $\SO_3\times \SO_3$ in $\SL_4$ acts on the 4-dimensional module via the tensor product representation of $A_1\times A_1$ on $E\otimes E$, where $E$ is the natural $2$-dimensional representation of $\SL_2$. To see that $\SL_2\times \SL_2$ acts reducibly (and hence $H$ as well) on all other non-trivial, $p$-restricted, irreducibles for $Y,$ we require a further argument (and some additional notation).

Let $T$ be a maximal torus of $\SL_2\times \SL_2$ with $T\subset T_Y.$ Moreover, choose a base $\Pi = \{\alpha,\beta\}$ of the root system of $\SL_2\times\SL_2$, and a base $\Pi(Y)=\{\alpha_1,\alpha_2,\alpha_3\}$ of the root system of $Y$, viewing $Y$ as $D_3$. Now it is straightforward to see that up to conjugation we may assume that $\alpha_1|_T = \beta-\alpha$, $\alpha_j|_{T} = \alpha$ for $j=2,3$. Now we apply \cite[6.1]{seitz} to see that $\lambda \in\{c\lambda_j,b\lambda_1+a\lambda_j, c\geq 1, b>0, a+b=p-1, j=2,3\}$. In case $\lambda=c\lambda_j$, $V|_{\SL_2\times\SL_2} = S^c(E\otimes E)$ which is easily seen to be reducible if $c>1$. In any case,  $V|_{\SL_2\times\SL_2}$ has a composition factor with highest weight $\lambda|_T$.  In the cases $\lambda=b\lambda_1+a\lambda_j$, we see that the weight $\lambda-\alpha_1$ restricts to $T$ as $\lambda|_T-\beta+\alpha$ and hence lies in a second composition factor of $V|_{\SL_2\times\SL_2}$.  The result then holds for $n=2$. 

Assume now that $n\geq 3$  and that the result holds for $Y=D_\ell$ with $\ell<n+1$. Assume as well that $p>2$; we will treat the case $p=2$ at the end of the proof. Let $W_1$ and $W_2$ be as in the previous proof so that $B_k$ acts on $W_1$ and $B_{n-k}$ acts on $W_2$. Let $U_1$, respectively $U_2$, be maximal totally isotropic subspaces of $W_1$, resp. $W_2$. Then $U_0=U_1\oplus U_2$ is an $n$-dimensional totally isotropic subspace of $W$. Let $P\subset Y$ be the preimage of the stabilizer in $\Isom(W_1)'\times\Isom(W_2)'$ of $U_0$ and $R$ the preimage in $Y$ of the stabilizer in $\Isom(W)$ 
of $U_0$. Then $P\subset R$, $R_u(P)\subset R_u(R)$, $(P/R_u(P))'\cong \SL(U_1)\times\SL(U_2)$, while $(R/R_u(R))'\cong\SL(U_0)=\SL_n$. In particular, the image of the Levi factor of $P$ in $R/R_u(R)$ stabilizes $U_1$ and so lies in a proper parabolic subgroup of $R/R_u(R)$, and hence can act irreducibly on no non-trivial $(R/R_u(R))'$-module.  Then the  above remarks and an application of the main proposition of \cite{smith} shows that $\lambda=a\lambda_n+b\lambda_{n+1}$, for some $a,b$.

Assume $\dim W_1\geq \dim W_2$, so $\dim W_1\geq 5$. (Recall that $\dim V \geq 8.$) Let $P_1$ be  the stabilizer in
 $\SO(W_1)'$ of a singular $1$-space. Then $P_1\times \SO(W_2)'$ is a proper parabolic subgroup of $\SO(W_1)'\times \SO(W_2)'$ and is contained in the image (under the natural projection $Y\to \SO(W)$)  of the stabilizer in $Y$ of this $1$-space.  As $P_1 = B_{k-1} R_u(P_1)$, we have the Levi factor $B_{k-1}B_{n-k}$ projecting into the Levi factor of type $D_n$. Another application of \cite{smith} and the induction hypothesis yield the result.

Turn now to the case $n>2$ and $p=2.$ In this case we have that the subgroup $B_k \times B_{n-k}$ stabilizes a non-singular  1-space and so lies in a subgroup of type $B_n$. Now the subgroup $B_{k}\times B_{n-k}$ is a maximal rank subgroup of $B_n$ and we can appeal to \cite[Theorem 4.1]{seitz} applied to the pair $(B_k\times B_{n-k},B_n)$ to see that the irreducible $KB_n$-module $V|_{B_n}$ is a twist of the spin module and hence we deduce that $\lambda=\lambda_n+\lambda_{n+1},$ $\lambda_n,$ or $\lambda_{n+1}.$ However, $L_Y(\lambda_n+\lambda_{n+1})|_{B_{n}}$ is not irreducible,
as the highest weight affords a twist of the spin module for $B_n$ and this is of dimension strictly less than the dimension of $L_Y(\lambda_n+\lambda_{n+1})$. So $\lambda=\lambda_n$ or $\lambda_{n+1}$ as claimed.
\end{proof}

Propositions~\ref{reduction} and \ref{semisimple} show that under the hypotheses of Proposition~\ref{Main_result_2}, 
either Proposition~\ref{Main_result_2}(a) holds, or $X$ acts irreducibly on the natural $KY$-module $W$, or $X\subset B_n$  is as in 
Proposition~\ref{reduction}\ref{reduction_2} with $k=0$. 
Theorem~\ref{Main_result_1} handles the latter situation in case $X=B_n$. The resolution of this case will follow from induction; see the end of this section.
Now for the case where $X$ acts irreducibly on $W$, we first show that $X$ must act tensor indecomposably.

\begin{prop}\label{tensorind} 
Let $Y,$ $X,$ $V$ be as in Proposition~$\ref{Main_result_2}$, and let $W$ be the natural $(2n+2)$-dimensional $KY$-module. If $W|_X$ is irreducible, then $W|_X$ is tensor indecomposable.
\end{prop}

\proof Suppose the contrary. Then write $W|_X=W_1^{(f_1)}\otimes\cdots\otimes W_t^{(f_t)}$, where $W_i$ is a $p$-restricted $KX$-module, $f_1,\dots,f_t$ are distinct $p$-powers, and $t\geq 2$. Then the criterion on the highest weight of self-dual modules shows that each $W_i$ carries an $X$-invariant non-degenerate bilinear form. Write $W=D\otimes F$, where each of $D$ and $F$ is a $KX$-module and the image of $X$ in $\Isom(W)'$ lies in $\Isom(D)'\circ\Isom(F)'$. Set $\dim D=d$ and $\dim F=f$ and assume $d\geq f\geq 2$, (and so at least one of $d$ and $f$ is even).

Extracting part of the argument given on page 76 of \cite{seitz}, we will now show that $Y$  contains a semisimple group inducing $\Isom(D)'\circ\Isom(F)'$. As pointed out above, $X$ stabilizes the product bilinear form on $W$. But then the irreducibility
of $X$ on $W$ forces this to be a scalar multiple of the form defining $Y$. Adjusting the form on $D$ if necessary, we may assume that the product form is precisely the form defining $Y$. This establishes the claim unless $p=2$. In this case, we have $\Isom(D)'\circ \Isom(F)'\subset \Sp(D)\circ\Sp(F)$. The latter group preserves a quadratic form $\mathcal Q$ on $W$, such that $\mathcal Q$ has the same polarization as the bilinear form on $W$ and also such that $\mathcal Q(x\otimes y)=0$ for all $x\in D$ and $y\in F$ (see \cite[Section 4.4]{KL}). By \cite[4.9]{Aschbacher}, $X$ fixes a unique quadratic form with prescribed bilinear form, and so this must necessarily be the form $\mathcal Q$ and we again have the claim. Hence we now have that $X$ lies in a closed subgroup $J$ of $Y$ which is the preimage in $Y$ of 
the group  $\Isom(D)'\circ\Isom(F)'$. In particular $J$ acts irreducibly on $V$ as well. Let $D_1$ be an isotropic $1$-space in $D$; then $D=D_1\oplus D_2\oplus D_3$, where $D_1^\perp=D_1\oplus D_2$, $D_2$ is non-degenerate, and $D_3$ is an isotropic $1$-space. Similarly decompose $F=F_1\oplus F_2\oplus F_3$. Now consider the flag of subspaces in $V:$
$$
0=V_0\subset V_1\subseteq V_2\subseteq V_3\subseteq V_4\subseteq V_5=V,
$$ 
where $V_1=D_1\otimes F_1$, $V_2=(D_1\otimes F_1^\perp)+(D_1^\perp\otimes F_1)$, $V_3=V_2^\perp$ and $V_4=V_1^\perp$. Then $V_1$ and $V_2$ are totally isotropic (totally singular if $p=2$), $\dim V_1=1=\dim(V/V_4)$ and $\dim(V_2/V_1)=\dim(V_4/V_3)=d+f-4$, and $\dim(V_3/V_2)=(d-2)(f-2)+2$.

Set $P_J$ to be the stabilizer in $J$ of the above flag, and $P_Y$ the stabilizer in $Y$ of this flag. Then $P_J$ is the product of the preimages of the parabolic subgroups $P_D$ and $P_F$ of $\Isom(D)$ and $\Isom(F)$ which are the stabilizers of the isotropic $1$-spaces $D_1$ and $F_1$, respectively. So $R_u(P_D)$ acts trivially on $D_1$, 
$D_1^\perp/D_1$ and $D/D_1^\perp$, and similarly for $R_u(P_F)$. On the other hand, $R_u(P_Y)$ is precisely the subgroup of $Y$ which acts trivially on $V_{j+1}/V_j$ for $0\leq j\leq 4$. One then sees that $R_u(P_J)\subset R_u(P_Y)$. Now $(P_Y/R_u(P_Y))'\cong (\Isom(V_3/V_2))'\circ\SL(V_2/V_1)$, and $(P_J/R_u(P_J))'\cong (\Isom(D_1^\perp/D_1))'\circ(\Isom(F_1^\perp/F_1))'$.

Assume for the moment that $f>2$; then $V_3/V_2$ is an orthogonal space of dimension at least 4. We note that the subspace $((D_1\otimes F_3) + V_2)/V_2$ is a non-zero singular subspace in $V_3/V_2$ left invariant by $(\Isom(D_1^\perp/D_1))'\circ(\Isom(F_1^\perp/F_1))'$ and so the projection of this latter group in $\Isom(V_3/V_2)$ is contained in a proper parabolic. It then follows from an application of \cite{smith} to the irreducible $KY$-module $V$ and the subgroup $J$ that the portion of the Dynkin diagram for $Y$ corresponding to the subgroup $\Isom(V_3/V_2)$ has zero labels, when representing $\lambda$ by a labeled diagram. But this contradicts our assumption on $\lambda$.

Consider now the case where $f=2$, so $\dim W=2n+2$ and $d=n+1$. Since we are assuming 
$Y=D_{n+1}$ with $n\geq 4$, we have $d\geq 5$. Now $\Isom(D_2)\circ \Isom(F_2)$ stabilizes the image of $D_1\otimes F_1^\perp$ in $V_2/V_1$ and so lies in a proper parabolic subgroup of $\Isom(V_2/V_1)$. Note that $\dim (V_2/V_1) = n-1$, and arguing as above we deduce that the nodes corresponding to $\SL(V_2/V_1)$ in the Dynkin diagram are labelled zero. So now we have $\lambda=a_1\lambda_1+a_n\lambda_n$, with $a_n\ne 0$. Moreover, the image of $X$ in $\Isom(W)'$ lies in $\Isom(D)'\circ\Isom(F)'= \Isom(D)'\circ\Sp_2$, which then implies that $D$ is even-dimensional and the latter group is $\Sp(D)\circ \Sp_2$, and acts irreducibly on $V$. The factor $\Sp(D)$ lies in the derived subgroup of an $A_n$-type Levi factor $L$ of $Y$, indeed is the naturally embedded  $\Sp_{n+1}$ subgroup of $A_n$. Moreover, $\Sp(D)$ acts homogeneously on $V$. Now $\lambda$ and $\lambda-\alpha_n$, or $\lambda-\alpha_{n-1}-\alpha_n-\alpha_{n+1}$,  depending on whether the root system of $L$ contains $\alpha_{n+1}$ or $\alpha_n$, afford the highest weights of irreducible summands of $V|_L$. It is then straightforward to see that the restrictions of these weights to the subgroup $\Sp(D)$ provide  non-isomorphic composition factors. This provides the final contradiction
in case $f=2$ and completes the proof of the result. 
\endproof

For the proof of Proposition~\ref{Main_result_2}, we continue with our consideration of the case where 
$X$ acts irreducibly and, by the previous result, tensor-indecomposably on the natural $KY$-module $W$. 
In particular, we may now assume that $X$ is a simple, proper, closed subgroup of $Y$. The hypotheses of 
Proposition~\ref{Main_result_2} then imply that $X$ is of type $B_m$ or of type $F_4$.
Moreover, we have the full set of hypotheses on the embedding of a parabolic subgroup  of $X$ in a parabolic subgroup of $Y$, in particular, with respect to the restriction of the highest weight $\lambda$ to a Levi factor. We will in fact show that $W|_X$ irreducible and the hypotheses of Proposition \ref{Main_result_2} are incompatible.

In each case, we will require some detailed information about the commutator series of an irreducible $KX$-module with 
respect to a fixed maximal parabolic subgroup. We start by considering:

\bigbreak

\noindent{\bf{Case 1: $X$ of type $B_m$, $m\geq 4$:}}\medbreak

Fix a maximal torus $T_X$ of $X$. Let $\Pi(X)=\{\beta_1,\dots,\beta_m\}$ be a base of the root system of $X$, $B_X$ the corresponding Borel subgroup containing $T_X$, and $\{\omega_1,\dots,\omega_m\}$ a set of fundamental dominant weights chosen with respect to the fixed base $\Pi(X)$. Set $s_i$ to be the reflection corresponding to the root $\beta_i$ for $1\leq i\leq m.$ For a torus $S$ of $X$, write $X(S)$ for the group of rational characters of $S$. Let $P_X$ be a maximal parabolic subgroup of $X$ corresponding to the subset $\{\beta_i :  2\leq i \leq m\}$, and containing the opposite Borel subgroup $B_X^-$. Let $Q_X=R_u(P_X)$ and $P_X=L_XQ_X$ for a Levi factor $L_X$ of $P_X.$ Let $\omega=\sum_{i=1}^m d_i\omega_i$ be a $p$-restricted dominant weight in $X(T_X)$; set $M=L_X(\omega)$ and assume $X$ preserves a non-degenerate quadratic from on $M$; let us denote this by $\mathcal Q$ and the associated bilinear form by $(\ ,\ ):M\times M\to K$.

For a unipotent group $J$ and a $KJ$-module $N$, we recall the standard notation $[N,J]$ for the subspace spanned by the set of vectors $n-xn$, where $x\in J$ and $n\in N$. We introduce an additional notation:
 set $[N,J^0]=N$ and set $[N,J^{i+1}]=[[N,J^i],J]$ for $i\geq 0$, so $[N,J^1]=[N,J]$.

\begin{defn}\label{def:levels}
Let $M=L_X(\omega)$ be as above.
\begin{enumerate}[(1)]
\item We will say a $T_X$-weight $\nu$ in $M$ has \emph{level} $i$ if $\nu=\omega-i\beta_1-\sum_{j=2}^m c_j\beta_j$ for some integers $i,c_j\geq 0$. 
\item Let $e(\omega)$ denote the maximum level of a weight. (When $\omega$ is fixed, we will simply write $e$.)

\end{enumerate}
\end{defn}

\begin{lem}\label{levels} 
Let $M = L_X(\omega)$. Then the following assertions hold.
\begin{enumerate} 
\item \label{levels_1} For $i\geq 0$, $[M,Q_X^i] = \sum M_\nu$, where the sum ranges over all $T_X$-weights $\nu$ of level at least $i.$
\item \label{levels_2} For $i\geq 0$, the $L_X'$-module $[M,Q_X^i]/[M,Q_X^{i+1}]$ is isomorphic to the module
$$
\sum_{\nu\in X(T_X)\mbox{\scriptsize{ of level }} i}M_\nu.
$$
\item \label{levels_3} The maximum level of weights in $M$ is $e=2(\sum_{i=1}^{m-1}d_i)+d_m$. If a weight $\nu$ is of level $j$, then $-\nu$ is of level $e-j$.
\item \label{levels_4} Let $S$ be a subtorus of $T_X$. For weights $\eta,\chi\in X(S)$, 
with $\eta\ne -\chi$, we have $V_\eta\subseteq V_\chi^\perp$. Moreover, if $\eta\ne 0$, then ${\mathcal Q}(u)=0$ for all $u\in V_\eta$.
\item \label{levels_5} The subspace $[M,Q_X^i]$ is totally singular for $i\geq (e+1)/2$.
\item \label{levels_6} If $e$ is even, then the subspace 
$$
\sum_{\nu\in X(T_X)\mbox{\scriptsize{ of level }} e/2} M_\nu
$$ 
is non-degenerate.
\item \label{levels_7} For $i\geq (e+1)/2,$ we have $[M,Q_X^i]^\perp = [M,Q_X^{e-i+1}].$
\item \label{levels_8} For all $i\geq (e/2)$, the quadratic form on $M$ induces an $L_X'$-invariant quadratic form on $[M,Q_X^i]/[M,Q_X^{i+1}]$. If the form is non-degenerate on this quotient space, then there exists $\nu$ of level $i$ such that $-\nu$ is also of level $i$. In particular, in this case, $d_m$ is even.
\end{enumerate}
\end{lem}

\begin{proof}
We start by proving \ref{levels_1}, arguing by induction on $i,$ the case $i=0$ being trivial by definition. Let $i\geq 1,$  and fix an ordering on $\Phi^+(X)$ such that any root containing $\beta_1$ is smaller than the others. By \cite[Section 27]{Humphreys2}, we clearly have that $[M,Q_X^i]\subseteq \sum M_\nu,$ where the sum ranges over all $T_X$-weights $\nu$ of level at least $i,$ and so it remains to show that $M_\nu \subseteq [M,Q_X^i]$ for all such $T_X$-weights $\nu.$ It is enough to show the latter for $T_X$-weights of level exactly $i.$ Let then $\nu$ be a $T_X$-weight of level $i$ in $M,$ and let $\gamma_1 \leq \ldots \leq \gamma_k $ be  roots in $\Phi^+(X)$ such that $\sum_{r=1}^k{\gamma_r}=\lambda-\nu.$ We  show that  $f_{\gamma_1}\cdots f_{\gamma_k}v^+ \in [M,Q_X^i],$ where $v^+$ is a maximal vector in $M.$ (This will be enough by Lemma \ref{Generators_for_weight_spaces_in_irreducibles}.) Set $w=f_{\gamma_2}\cdots f_{\gamma_k}v^+.$ Then $w$ has level $i-1,$ as $\gamma_1$ involves $\beta_1$ thanks to our choice of ordering on $\Phi^+,$ and hence $w\in [M,Q_X^{i-1}]$ by induction. Also, we have
$$
x_{-\gamma_1}(1)w - w \in  f_{\gamma_1}\cdots f_{\gamma_k}v^+ + \sum_{r=1}^{\infty}{M_{\nu -r\gamma_1}}, 
$$
and hence $x_{-\gamma_1}(1)w - w$ has non-zero coefficient of $f_{\gamma_1}\cdots f_{\gamma_k}v^+,$ and all other terms in the sum are weight vectors of weights different from $\nu.$ One then deduces that $[M,Q_X^i]$ must contain $f_{\gamma_1}\cdots f_{\gamma_k}v^+$ as desired, showing that \ref{levels_1} holds.

The statement of \ref{levels_2} now follows by induction on $i$.

For \ref{levels_3}, set $w_0$ to be the longest word of the Weyl group of $X$. Writing the $\omega_i$ in terms of the simple roots $\beta_j$, we see that $w_0(\omega) = -\omega=\omega-2(d_1+d_2+\cdots+d_{m-1} + \frac{1}{2}d_m)\beta_1-\sum_{i=2}^m c_i\beta_i$ for some non-negative integers $c_i$, giving the result. For the final statement, suppose that $\nu=\omega-j\beta_1-\sum_{i=2}^m c_i\beta_i$, then  $-\nu=\omega-2\omega+j\beta_1+\sum_{i=2}^mc_i\beta_i =  \omega-(e-j)\beta_1-\sum_{i=2}^mb_i\beta_i$, for some non-negative integers $b_i$.

The statement of \ref{levels_4} is standard, and \ref{levels_5} and \ref{levels_6} follow from \ref{levels_4}.

For \ref{levels_7}, note that by \ref{levels_1} and \ref{levels_4}, we have $[M,Q_X^{e-i+1}]\subseteq [M,Q_X^i]^\perp$. Then the result follows from a  dimension argument using \ref{levels_2} and \ref{levels_3}.

For \ref{levels_8}, we note that for $i\geq (e/2)$, $[M,Q_X^{i+1}]$ is totally singular  and so the given quadratic form induces an $L_X'$-invariant form on the quotient. Now if $p\ne 2$, the second statement follows directly from \ref{levels_4}. If $p=2$, and $[M,Q_X^i]/[M,Q_X^{i+1}]$ is odd-dimensional, so that the bilinear form has a 1-dimensional non-singular radical, then \ref{levels_4} implies that the radical is generated by a vector of weight 0. So the weight $\nu=0$ satisfies the given condition. For the claim about the parity of $d_m$,  we recall that for $\nu$ of level $j$, $-\nu$ is of level $e-j$. So we deduce that $e$ is even, and so $d_m$ as well.
\end{proof}

In case $e$ is even, we will consider a certain $T_X$-weight at level $(e/2)$, namely the weight  
$$
\mu: = \omega-(e/2)\beta_1-(e/2-d_1)\beta_2 - 
\cdots - (d_{m-1}+(d_m/2))\beta_{m-1}-(d_m/2)\beta_m.
$$ 
Note that 
\begin{equation}
\mu = -(e/2)\omega_1+\sum_{i=2}^{m-1}d_{i-1}\omega_i+(2d_{m-1}+d_m)\omega_m.
\label{mu}
\end{equation}

\begin{lem}\label{lem_mu} 
Assume $d_m$ is even. For each $0\leq j\leq d_m$, there exists a unique  weight $\nu_j\in X(T_X)$, of  level $j+\sum_{i=1}^{m-1}d_i$, such that $\nu_j|_{T_X\cap L_X'} = \mu|_{T_X\cap L_X'}$. In addition, each $\nu_j$ is of multiplicity 1 and 
if $\nu\in X(T_X)$ with $\nu|_{T_X\cap L_X'} = \mu|_{T_X\cap L_X'}$, then $\nu=\nu_j$ for some $j$. In particular, there exists an odd number of weights $\nu$ such that $\nu|_{T_X\cap L_X'} = \mu|_{T_X\cap L_X'}$.
\end{lem}

\proof Let $\nu\in X(T_X)$ such that $\nu|_{T_X\cap L_X'} = \mu|_{(T_X\cap L_X')}$. Then  $\nu=\omega-\sum_{i=1}^m c_i\beta_i,$ and the $c_i$ must satisfy the equations
$$
d_{j+1}+c_j-2c_{j+1}+c_{j+2} = d_j\mbox{ for }1\leq j\leq m-2,\mbox{ and}
$$
$$
d_m+2c_{m-1}-2c_m = 2d_{m-1}+d_m.
$$
Solving these equations leads to the relations $c_j=c_m+\sum_{i=j}^{m-1}d_i$, for $1\leq j\leq m-1$. Now note as well that for $\nu$ as given, $s_{m-1}s_{m-2}\cdots s_{1}\nu = \omega-c_m\beta_m$. Hence $\nu$ has multiplicity at most 1 and has multiplicity 1 if and only if $c_m\leq d_m$. So for all $0\leq j\leq d_m$, there exists a weight $\nu_j$, of multiplicity 1, and of level $j+\sum_{i=1}^{m-1}d_i$ whose restriction to $T_X\cap L_X'$ is equal to $\mu|_{(T_X\cap L_X')}$.
\endproof

\begin{prop}\label{Bm-case} 
Assume the hypotheses of  Proposition \textnormal{\ref{Main_result_2}} with  $X=B_m$. Then $X$ acts reducibly on the  
natural $KY$-module.
\end{prop}

\proof Suppose the contrary, that is, let $W$ be the natural $KY$-module and suppose that  $W|_X$ is irreducible. 
Then by 
Proposition~\ref{tensorind}, $W|_X$ is tensor indecomposable. If $W|_X$ is not $p$-restricted, then we replace $W|_X$ by a $p$-restricted representation $\rho$ of a  simply-connected cover $\tilde X$ of $X$, and 
consider $W|_{\rho(\tilde X)}$.
For the purposes of the argument, we may replace $X$ by $\rho(\tilde X)$. Let $W=L_X(\omega)$ for $\omega=\sum_{i=1}^md_i\omega_i$, a $p$-restricted dominant weight.

 As in the hypotheses of Proposition~\ref{Main_result_2}, let $P_Y$ be a parabolic subgroup of $Y$ containing $P_X$ with $Q_X\subset R_u(P_Y) = Q_Y$. Fix a maximal torus $T_Y$ of $P_Y$ with $T_X\subset T_Y$. We choose a base $\Pi(Y)=\{\alpha_1,\dots,\alpha_{n+1}\}$ of the root system of $Y$ so that $P_Y$ contains the opposite Borel $B_Y^-$ subgroup with respect to this base.

Now let $\{0\}=W_0\subsetneq W_1\subsetneq \cdots\subsetneq W_{t-1}$ be a flag of totally singular subspaces, such that $P_Y$ is the full stabilizer in $Y$ of this flag; so $P_Y$ is the stabilizer of the flag
$$
\{0\}=W_0\subsetneq\cdots \subsetneq W_{t-1}\subseteq W_{t-1}^\perp
\subsetneq W_{t-2}^\perp\cdots\subsetneq W_0^\perp=W.
$$

Setting $W_{t-1+j}=W_{t-j}^\perp$ for $0\leq j\leq t,$ we have the flag  
$$
\{0\} = W_0\subsetneq\cdots\subsetneq W_{t-1}\subseteq W_t\subsetneq W_{t+1}
\cdots\subsetneq W_{2t-1}=W,
$$
of which $P_Y$ is the full stabilizer. By our hypotheses on $P_Y$, we have that $W_t/W_{t-1}$ is an orthogonal space of dimension at least 8 (so $W_{t-1}\ne W_t$) and we have $L_t=(\Isom(W_t/W_{t-1}))'$, a group of type $D_m$ for $m\geq 4$. Note that $W_j = [W,Q_Y^{2t-1-j}]$, for all $j\geq 0$.

Recall that by hypothesis, $L_X'$ is of type $B_{m-1}$. If $m>4$, $L_X'$ is the stablizer in $L_t$ of a non-singular $1$-space of $W_t/W_{t-1}$. Note that if $m=4$, so $L_X'$ is of type $B_3$ and $(\Isom(W_t/W_{t-1}))'=D_4$, it may be that $L_X'$ acts irreducibly on $W_t/W_{t-1}$ as a spin module. We will take care to consider the latter possibility in what follows.

We need to see how the two flags which are stabilized by $P_X$ are related, i.e. the flag of subspaces $\{[W,Q_X^j]\}_{j\geq 0}$ and the flag of subspaces $\{W_i\}_{i\geq 0}$. We first show that $e$ (as defined in Lemma \ref{levels} \ref{levels_3}) is even. Indeed, suppose the contrary. Then $P_X$ lies in the stabilizer in $Y$ of $[W,Q_X^{\frac{e+1}{2}}]$, a totally singular subspace of dimension $\frac{1}{2}\dim W$ (using Lemma \ref{levels} \ref{levels_7}). Set $P=QL$ to be this stabilizer, a maximal parabolic subgroup of $Y$ with Levi factor $L$ of type $A_n$ and $Q=R_u(P)$. Consider now the action of $L_X'$ and $L'$ on the commutator quotient $V/[V,Q]$. By \cite{smith}, each group acts irreducibly. But given that the restriction of $\lambda$ to $T_Y\cap L'$ has at least two non-zero labels, we see that there are no compatible configurations coming from Table 1 of \cite{seitz}. (Here we are invoking Hypothesis \ref{inductive_hypothesis} (ii).) Hence $e$ is even as claimed. In particular, the weight $\mu$ of (\ref{mu}) exists and Lemma \ref{lem_mu} applies.

Choose $s$ minimal such that $[W,Q_X^{(e/2)}]\subseteq W_s$. Then $s\geq t$ because $[W,Q_X^{(e/2)}]$ is not totally singular by Lemma \ref{levels} \ref{levels_6}. If $s=t$, then $[W,Q_X^{(e/2)}]\subseteq W_t$ and $[W,Q_X^{(e/2)}]$ is not contained in $W_{t-1}$. If $W_\mu\not\subset W_{t-1}$, then
 $\mu|_{L_X'\cap T_X}$ occurs as a weight in the quotient module $(W_t/W_{t-1})|_{L_X'}$.
But the only dominant $T_X\cap L_X'$ weights in this module are either the zero weight or the first fundamental dominant weight, or $m=4$ and $W_t/W_{t-1}$ is the spin module for $L_X'$. But comparing each of these weights and  $\mu|_{T_X\cap L_X'}$ (see (\ref{mu})),  we see that either $\omega=0$, or $\omega=\omega_1$, or $m=4$ and $\mu$ affords the spin module. The first two possibilities are not consistent with our assumption that $X$ acts 
irreducibly on the natural module for $Y$.  In the last case, when $m=4$ and $\mu$ affords the spin module for $L_X'$, we deduce that $2d_{m-1}+d_m=1$, contradicting the fact that $d_m$ is even. Hence we have $W_\mu\subseteq W_{t-1}$. Choose $r$ minimal such that $W_\mu\subseteq W_r$. Then $\mu$ occurs in the quotient module $W_{r}/W_{r-1}$, which is 
dual to $W_{r-1}^\perp/W_r^\perp$, and so the weight $\mu|_{T_X\cap L_X'}$ occurs here as well. But by Lemma \ref{lem_mu}, there are an odd number of weights, each of which is of multiplicity $1,$ whose restriction to $T_X\cap L_X'$ is $\mu$ and so there exists $\nu\in X(T_X)$ such that $\nu|_{T_X\cap L_X'} = \mu$ and $W_\nu\subseteq W_t$ with $W_\nu\not\subset W_{t-1}$. Now we conclude as above by comparing $\nu|_{T_X\cap L_X'}$ and the weights occurring in $(W_t/W_{t-1})|_{L_X'}$.

Now consider the case that $s>t$. Then $[W,Q_X^{(e/2)+j}]\subseteq W_{s-j}$ is an $L_X$-invariant subspace. In particular, taking $j=s-t$ we have a totally singular $L_X'$-invariant subspace of $W_t$. But since the image of $L_X'$ in $\Isom(W_t/W_{t-1})$ does not lie in a proper parabolic we see that in fact $[W,Q_X^{(e/2)+s-t}]\subseteq W_{t-1}$. Then we have $W_{t-1}^\perp\subset [W,Q_X^{(e/2)+s-t}]^\perp$ by Lemma \ref{levels} \ref{levels_7}, which implies that
$W_t\subseteq [W,Q_X^{(e/2)-s+t+1}]$.

Now consider the inclusions 
$$
[W,Q_X^{(e/2)+s-t}]\subseteq W_{t-1}\subseteq W_t\subseteq [W,Q_X^{(e/2)-s+t+1}].
$$
We again ask where our weight $\mu$ occurs. If $W_\mu\subseteq W_t$ and $W_\mu\not\subset W_{t-1}$, we can argue as before. If $W_\mu\cap W_t=\{0\}$ (recall $\dim W_\mu=1$), then $\mu$ occurs as a weight in the quotient module $[W,Q_X^{(e/2)-s+t+1}]/W_t$ and therefore must also occur in the dual, which is the quotient module $W_{t-1}/[W,Q_X^{(e/2)+s-t}]$ and now we argue as above to see that there exists $\nu$ whose restriction to $T_X\cap L_X'$ is $\mu$ and occurring in the quotient $W_t/W_{t-1}$, leading to a  contradiction as above. This completes the proof of the proposition.
\endproof

\noindent{\bf{Case 2:  $X=F_4$:}} \medbreak

\begin{prop}\label{F4-case} 
Assume the hypotheses of  Proposition \textnormal{\ref{Main_result_2}} with  $X=F_4$. Then the  natural $KY$-module $W$ is a reducible $KX$-module.
\end{prop}

\proof 
Fix a maximal torus $T_X$ of $X$. Let $\Pi(X)=\{\beta_1,\dots,\beta_4\}$ be a base of the root system of $X$,  $B_X$ the corresponding Borel subgroup containing $T_X$, and $\{\omega_1,\dots,\omega_4\}$ a set of fundamental dominant weights 
chosen with respect to the fixed base $\Pi(X)$. 

Suppose  that $W|_X$ is irreducible. Then by Proposition~\ref{tensorind}, we have $W|_X$ 
tensor-indecomposable and so without loss of  generality we will assume the highest weight $\omega$ is $p$-restricted, that is $\omega=\sum_{i=1}^4 d_i\omega_i$, with $d_i<p$ for all $i$. We will treat the case $p=2$ separately below. 
$$
\text{Assume for now that } p>2.
$$  

Let $P_X$ be the maximal parabolic subgroup of $X$ corresponding to the simple root $\beta_1$, containing the opposite Borel subgroup.  We extend the notion of ``level'' to this setting and rely on \cite[2.3]{seitz}. Note that the maximum level of weights with respect to this parabolic is $e:=2(2d_1+3d_2+2d_3+d_4)$. Consider the flag of totally singular subspaces 
$$
\{0\}=[W,Q_X^{e+1}]\subseteq [W,Q_X^e]\subseteq\cdots \subseteq [W,Q_X^{(e/2)+1}].
$$ 
Then $P_X$ lies in the stabilizer of this flag, which is the stabilizer of the full flag 
$$
\{0\}=[W,Q_X^{e+1}]\subseteq [W,Q_X^e]\subseteq\cdots \subseteq [W,Q_X^{(e/2+1}]\subseteq[W,Q_X^{(e/2)}]\subseteq\cdots\subseteq [W,Q_X]\subseteq W.
$$
The quotient $[W,Q_X^{(e/2)}]/[W,Q_X^{(e/2)+1}]$ is a non-degenerate subspace and we claim that it is non-trivial of dimension at least 6, which then implies that the Levi factor $L_X'$ of type $C_3$ projects non-trivially into $\Isom([W,Q_X^{(e/2)}]/[W,Q_X^{(e/2)+1}])$, a group of type $D_\ell$ with $\ell\geq 4$ (as there is no non-trivial morphism from $C_3$ into $D_\ell$ for $\ell\leq 3$. Indeed, consider the weight
$$
s_1s_2s_3s_4s_2s_3(\omega-d_1\beta_1-d_2\beta_2) = \omega-\tfrac{e}{2}\beta_1-\tfrac{e}{2}\beta_2-(2d_1+4d_2+3d_3+d_4)\beta_3-(2d_2+d_3+d_4)\beta_4.
$$ 
This weight is of level $(e/2)$, occurs in $W$ by \cite{premet},  and affords the dominant $T_X\cap L_X'$ weight $(2d_2+d_3)\omega_2+d_4\omega_3+(2d_1+d_3)\omega_4$, where here we write $\omega_i$ for $\omega_i|_{T_X\cap L_X'}$. Hence $\dim([W,Q_X^{(e/2)}]/[W,Q_X^{(e/2)+1}])\geq 6$ as claimed.

Now consider the action of $X$ on the irreducible module $V=L_Y(\lambda)$ where $\lambda$ 
satisfies the  congruence relations. Since $C_3$ projects non-trivially into the $D_\ell$ factor of $P_Y$, we must have that $C_3$ acts irreducibly on the $D_\ell$ module with the highest weight as given. But there is no such example in \cite[Table 1]{seitz}, giving the desired contradiction. (Here we invoked Hypothesis \ref{inductive_hypothesis} (iii).) 

Finally, consider the case where $p=2$. Since $W|_X$ is tensor-indecomposable, \cite[1.6]{seitz} implies that $\omega=d_1\omega_1+d_2\omega_2$ or $d_3\omega_3+d_4\omega_4$. Moreover, using the graph automorphism of $F_4$, it suffices to consider the second situation, in which case \cite[2.3]{seitz} still applies. Using \cite{luebeck}, we see that the weight lattice of $W|_X$ is as in characteristic 0, in particular, we can exhibit as above in each case a non-zero dominant weight of level $e/2$ and the above argument goes through. If $\omega = \omega_3$, when $e=4$, we take $\omega-2\beta_1-2\beta_2-3\beta_3-\beta_4$; if $\omega=\omega_4$ and $e=2$, take $\omega-\beta_1-\beta_2-\beta_3-\beta_4$; if $\omega=\omega_3+\omega_4$ and $e=6$, take $\omega-3\beta_1-3\beta_2-4\beta_3-2\beta_4$. 
\endproof

\noindent\begin{proof}[{\bf{Proof of Proposition~$\ref{Main_result_2}$}}] We can now complete the proof of Proposition~\ref{Main_result_2}. Let $X$, $Y$ and $V$ be as in the statement of the proposition and assume Hypothesis~\ref{inductive_hypothesis} for all embeddings $H\subset G$ such that $\rank G < \rank Y.$
By Propositions~\ref{reduction}, \ref{semisimple}, \ref{tensorind}, \ref{Bm-case}, \ref{F4-case}, we have that either Proposition~\ref{Main_result_2}(a) holds, or $X$ acts reducibly on $W$ and
 $X\subset B_n\subset Y$. If $X=B_n$, Proposition~\ref{Main_result_2}(b) holds; so assume $X\subsetneq B_n$. The restriction of $V$ to $B_n$ affords a $p$-restricted irreducible $KB_n$-module with highest weight having at least three nonzero coefficients when expressed in terms of the fundamental dominant weights for $B_n$. Moreover, $X$ has a Levi subgroup of type $B_m$, by hypothesis. We now refer to \cite[Table 1]{seitz}, invoking Hypothesis~\ref{inductive_hypothesis} (i), to see that there are no such 
irreducible triples $(X,B_n,\lambda|_{T_{B_n}})$, where $T_{B_n}$ is a maximal torus of $B_n$ lying in the maximal torus $T_Y$. 
This completes the proof of the result.\end{proof}

\section{Proof of Theorem \ref{Main_result_3}}\label{exceptional}

Let $ Y$ be a simply connected, simple algebraic group of type $E_n$, $n=6,7,8$, defined 
over $K.$ We start by considering a maximal, closed, connected semisimple subgroup $X$ of $Y$, satisfying the hypotheses of 
Theorem~\ref{Main_result_3}, namely, $X$ has a proper parabolic subgroup whose Levi factor is of type $B_m$, for some $m\geq 3$. 
 Referring to \cite[Theorem 1]{liebeck-seitz}, we see that we must consider the following pairs $(X,Y)$:
\begin{itemize}
\item $(F_4,E_6)$,
\item $(A_1F_4,E_7)$,
\item $(G_2F_4,E_8)$.
\end{itemize}

We start by dealing with the two latter possibilities.

\begin{prop}\label{other_embeddings} The maximal subgroup $A_1F_4\subset E_7$ acts reducibly on all non-trivial irreducible 
$KE_7$-modules. The maximal subgroup $G_2F_4\subset E_8$ acts  reducibly on all non-trivial irreducible $KE_8$-modules.
\end{prop}

\begin{proof}We assume the contrary. Let $X=M_1F_4\subset Y$ be a maximal subgroup of $Y=E_n$, for $M_1=A_1$, respectively $G_2$ and $n=7$, respectively 8. The factor $F_4$ is embedded in the usual way as a maximal subgroup of an $E_6$ Levi factor of $Y$. Assume that $X$ acts irreducibly on $V=L_Y(\lambda)$ where $\lambda$ is a non-zero $p$-restricted dominant weight. Adopting the usual notation as in previous results, we set $\lambda=\sum_{i=1}^n a_i\lambda_i,$ where $\{\lambda_1,\dots,\lambda_n\}$ are the fundamental dominant weights with respect to a fixed choice of base for the root system of $Y$. Since $X$ acts irreducibly on $V$, $F_4$ acts homogeneously. Let $\{\omega_1,\omega_2,\omega_3,\omega_4\}$ be a set of fundamental dominant  weights for $F_4$ again with respect to a fixed choice of base of the root  system which is compatible with the given choice for $Y$. By \cite{smith} we have one $F_4$ composition factor of $V$ with highest weight $\omega=a_2\omega_1+a_4\omega_2+(a_3+a_5)\omega_3+(a_1+a_6)\omega_4$. The homogeneity of $V|_{F_4}$ implies that $a_7=0$, and $a_8=0$ if $n=8,$ as otherwise $\lambda-\alpha_7$ (resp. $\lambda-\alpha_7-\alpha_8$) would afford the highest weight of a composition factor different from $L_X(\omega).$ Now let $i_0$ be maximal such that $a_{i_0}\ne 0$ (such exists since $\lambda\ne 0$). If $i_0=2$, then $\lambda-\alpha_2-\alpha_4-\alpha_5-\alpha_6-\alpha_7$ affords an $F_4$ composition factor of $V$ not isomorphic to that already given. Hence $i_0\ne 2$. If $i_0=1$, then $\lambda-\alpha_1-\alpha_3-\alpha_4-\alpha_5-\alpha_6-\alpha_7$ has the same property. For all other cases, we take $\lambda-\sum_{i=i_0}^7\alpha_i$, which again affords an $F_4$-composition factor with highest weight different from $\omega$. This is the final contradiction.
\end{proof}

The remainder of this section is devoted to the proof of  the following result, first proven by Testerman \cite[Theorem 5.0 (i)]{test} under the assumption that \cite[Table $1,$  $\mbox{IV}1,$ $\mbox{IV}1'$]{seitz} formed a complete family of irreducible triples for the usual embedding $B_n\subset D_{n+1}$.

\begin{thm}\label{F_4_in_E_6}
Let $Y$ be a simply connected, simple algebraic group of type $E_6$ over $K$ and let $X$ be the subgroup of type $F_4,$ embedded in $Y$ in the usual way. Also let $V=L_Y(\lambda)$ be a non-trivial, irreducible  $KY$-module having $p$-restricted highest weight $\lambda\in X^+(T_Y).$ Assume Hypothesis~$\ref{inductive_hypothesis}$ for all embeddings $H\subset G$ with $\rank(G)<\rank(Y).$ Then $X$ acts irreducibly on $V$ if and only if one of the following holds, where we give $\lambda$ up to graph automorphisms.
\begin{enumerate}
\item $\lambda=(p-3)\lambda_1,$ with $p>3.$
\item $\lambda=\lambda_1+(p-2)\lambda_3,$ with $p>2.$ 
\end{enumerate}
\end{thm}

The proof relies on some detailed knowledge of the structure of certain Weyl modules for $E_6$ and $F_4.$ Section \ref{Tools} below provides some results based on the Jantzen $p$-sum formula. In Section \ref{Computing_multiplicities}, we apply the methods from Section \ref{Tools} to various Weyl modules in order to obtain insight on their structure, as well as information on certain weight multiplicities in the corresponding irreducible quotient. These results shall then prove useful in Section \ref{3rd_result:conclusion}, in which we conclude by showing  that Theorem \ref{F_4_in_E_6} holds. Finally, at the end of the section, we see how these results lead to a proof of Theorem \ref{Main_result_3}.

\subsection{Understanding Weyl modules}\label{Tools}     

Let $G$ be a semisimple algebraic group over $K,$ let $B$ be a Borel subgroup of $G$ containing a fixed maximal torus $T.$ Let $\Pi = \{\alpha_1,\ldots,\alpha_n\}$ denote the corresponding base for the root system $\Phi=\Phi^+\sqcup \Phi^-$ of $G,$ and let $\lambda_1,\ldots,\lambda_n$ denote the corresponding fundamental dominant weights for $T.$ Let $\rho$ denote the half-sum of all positive roots in $\Phi,$ or equivalently, the sum of all fundamental dominant weights. Also for $\lambda,\mu\in X^+(T)$ such that $\mu \prec \lambda,$ define 
$$
\mbox{d}(\lambda,\mu)=2(\lambda+\rho,\lambda-\mu)-(\lambda-\mu,\lambda-\mu).
$$
The following corollary to the strong linkage principle \cite{Andersen} gives a necessary condition for $\mu$ to afford the highest weight of a $KG$-composition factor of $V_G(\lambda),$ in the case where $p>2$ and $G$ is  not of type $G_2.$ We refer the reader to \cite[6.2]{seitz} for a proof.

\begin{prop}\label{Corollary to The linkage principle}
Assume $p>2$ and let $G$ be a simple algebraic group of type different from $G_2.$  Also let $\lambda$ and $\mu$ be as above, and assume the inner product on $\Z \Phi$ is normalized so that long roots have length $1.$ If $\mu $ affords the highest weight of a composition factor of $V_G(\lambda),$ then 
$$
2\textnormal{d}(\lambda,\mu)\in p\Z.
$$
\end{prop}

Let $\{e^{\mu}\}_{\mu\in X(T)}$ denote the standard basis of the group ring $\Z[X(T)]$ over $\Z.$ The Weyl group $\mathscr{W}$ of $G$ acts on $\Z[X(T)]$ by $we^{\mu}=e^{w\mu},$ $w\in \mathscr{W},$ $\mu\in X(T),$ and we write $\Z[X(T)]^{\mathscr{W}}$ to denote the set of fixed points. The \emph{formal character} of a given $KG$-module $V$ is the linear polynomial $\ch V \in \Z[X(T)]^{\mathscr{W}}$ defined by
$$
\ch V = \sum_{\mu\in X(T)}{\m_V(\mu) e^{\mu}}.
$$
Also, for  $\lambda\in X^+(T),$  we write $$\chi(\lambda)= \ch V_G(\lambda) = \ch H^0(\lambda)$$ (see \cite[II, 2.13]{Jantzen}, for instance). The  \emph{Jantzen $p$-sum formula}  \cite[II, Proposition 8.19]{Jantzen} yields the existence of a filtration $V_G(\lambda)=V^0 \supsetneq V^1 \supseteq \ldots \supseteq V^k \supseteq V^{k+1} =0$ of $V_G(\lambda),$ such that $V^0/V^1\cong L_G(\lambda)$ and such that the sum $\sum_{r=1}^k{\ch V^r} \in \mathbb{Z}[X(T)]^{\mathscr{W}},$ denoted $\nu^c(T_\lambda),$ satisfies certain properties (\emph{loc. cit.}). Throughout this section, we call such a filtration a \emph{Jantzen filtration} of $V_G(\lambda).$ Moreover,  since  $\{\chi(\lambda)\}_{\lambda\in X^+(T)}$ forms a $\Z$-basis of $\Z[X(T)]^{\mathscr{W}} $  (see \cite[II, Remark 5.8]{Jantzen}), there exists $(a_\nu)_{\nu\in X^+(T)}\subset \mathbb{Z}$  such that   
\begin{equation}
\nu^c(T_\lambda)=\sum_{\nu \in X^+(T)}{a_\nu \chi(\nu)}.
\label{nu^c(T_lambda)_in_terms_of_chi(nu)'s}
\end{equation}

Consider a $T$-weight $\mu \in X(T)$ with $\mu \prec \lambda.$ Following \cite[Section 3.2]{Cavallin2}, we now define a ``truncated'' version of $\nu^c(T_\lambda),$ which shall prove useful in computations, by setting
\begin{equation} 
\nu_\mu^c(T_\lambda)= \sum_{\substack{\nu \in X^+(T) \\ \mu \preccurlyeq \nu \prec \lambda}}{a_\nu \chi_\mu (\nu)},
\label{nu_mu^c(T_lambda)_in_terms_of_chi_mu(nu)'s}
\end{equation} 
where the $a_\nu$ $(\nu\in X^+(T))$ are as in \eqref{nu^c(T_lambda)_in_terms_of_chi(nu)'s}, and where for every $\nu \in X^+(T),$ we have $\chi_\mu(\nu)=  \sum_{\substack{ \eta \in X^+(T) \\ \mu \preccurlyeq \eta \preccurlyeq\nu}}{[V_G(\nu),L_G(\eta)]\ch L_G(\eta)}.$  Finally, the latter decomposition yields
\begin{equation}
\nu_\mu^c(T_\lambda)= \sum_{\substack{\xi \in X^+(T) \\ \mu \preccurlyeq \xi \prec \lambda}}{b_\xi \ch L_G(\xi)},
\label{nu_mu^c(T_lambda)_in_terms_of_ch(xi)'s}
\end{equation}
for some $b_\xi \in \Z.$ The following proposition shows how \eqref{nu_mu^c(T_lambda)_in_terms_of_ch(xi)'s} can be used in order to determine the possible composition factors of $V_G(\lambda),$ together with an upper bound for their multiplicity. We refer the reader to \cite[Proposition 3.6]{Cavallin2} for a proof.

\begin{prop}\label{Insight_on_possible_composition_factors} 
Let $\lambda\in X^+(T)$ and consider a $T$-weight $\mu \prec \lambda.$ Also let $\xi \in X^+(T)$ be a dominant weight such that $\mu \preccurlyeq \xi \prec \lambda.$ Then $\xi$ affords the highest weight of a composition factor of $V_G(\lambda)$ if and only if $b_\xi\neq 0$ in \eqref{nu_mu^c(T_lambda)_in_terms_of_ch(xi)'s}. Also  $[V_G(\lambda),L_G(\xi)]\leq b_\xi.$
\end{prop}

For $\nu\in X^+(T),$ we call the coefficient $a_\nu$ in \eqref{nu_mu^c(T_lambda)_in_terms_of_chi_mu(nu)'s} the \emph{contribution} of $\nu$ to $\nu^c_\mu(T_\lambda),$ and we say that $\nu$ \emph{contributes} to $\nu^c_\mu(T_\lambda)$ if its contribution is non-zero. Now applying Proposition \ref{Insight_on_possible_composition_factors} for specific weights $\mu,\xi$ with $\mu\preccurlyeq \xi \prec \lambda$ requires the knowledge of the contribution of $\nu$ to $\nu_\mu^c(T_\lambda)$ for each dominant $T$-weight $\mu \preccurlyeq \nu \prec \lambda.$  In certain cases, knowing whether or not a given $T$-weight contributes to $\nu^c_\mu(T_\lambda)$  can be easily determined, as the following result shows. We refer the reader to \cite[Lemma 3.7]{Cavallin2} for a proof.

\begin{lem}\label{no_contribution_for_certain_weights}
Let $\lambda,$ $\mu$ and $\nu$ be as above, with $\nu$ maximal, with respect to the partial order   $\preccurlyeq$, such that $\nu$ contributes to $\nu^c_{\mu}(T_{\lambda}).$ Then $\nu$ affords the highest weight of a composition factor of $V_G(\lambda).$  
\end{lem} 

Fix $\nu\in X^+(T),$ and recall from \cite[Planches I-IV]{Bourbaki} the description of the simple roots and fundamental dominant weights for $T$ in terms of a basis $\{\varepsilon_1,\ldots,\varepsilon_{d_\Phi}\}$ for a Euclidean space $E$ of dimension $d_\Phi.$   For  $\alpha \in \Phi^+$ and $r\in \mathbb{Z}_{\geq 0}$ such that $1<r<\langle \lambda+\rho, \alpha\rangle$, we write $\lambda+\rho-r\alpha=a_1\varepsilon_1+\cdots+a_{d_\Phi}\varepsilon_{d_\Phi},$ as well as $\nu +\rho =b_1\varepsilon_1+\cdots+b_{d_\Phi}\varepsilon_{d_\Phi},$ following the ideas of \cite{McNinch}. Also following \cite[Section 3.2]{Cavallin2}, we set $A_{\alpha,r}=(a_j)_{j=1}^{d_\Phi}\in \mathbb{Q}^{d_\Phi}$ and $B_{\nu}= (b_j)_{j=1}^{d_\Phi}\in \mathbb{Q}^{d_\Phi}.$  The action of the Weyl group $\mathscr{W}$ of $G$ on the basis $\{\varepsilon_1,\ldots,\varepsilon_{d_\Phi}\},$ described in \cite[Planches I-IV]{Bourbaki}, extends to an action of $\mathscr{W}$ on $\mathbb{Q}^{d_\Phi}$ in the obvious way. (We write $w\cdot A$ for $w\in \mathscr{W},$ $A\in \mathbb{Q}^{d_\Phi}.$) In addition, define the \emph{support} of an element $z\in \Z\Phi$ to be the subset $\mbox{supp}(z)$ of $\Pi$ consisting of those simple roots $\alpha$ such that $c_{\alpha}\neq 0$ in the decomposition $z=\sum_{\alpha\in \Pi}{c_{\alpha}\alpha}.$ Finally, for  $w\in \mathscr{W},$ we write $\det (w)$ for the determinant of $w$ as an invertible linear transformation of $X(T)_{\R}=X(T)\otimes_\Z \R.$ The following result is  our main tool for determining the contribution of $\nu$ to $\nu^c_\mu(T_\lambda),$ for each weight $\nu\in X^+(T)$ with $\mu\preccurlyeq \nu \prec\lambda.$ We refer the reader to \cite[Theorem 3.8]{Cavallin2} for a proof.

\begin{thm}\label{How_to_determine_contributions}     
Let $\lambda\in X^+(T),$ and consider a weight $\mu\in X(T)$ with $\mu\prec\lambda.$ Let $\nu\in X^+(T)$ be such that $\mu\preccurlyeq \nu\prec\lambda.$ Write $I_\nu=\{(\alpha,r)\in \Phi^+\times [2,\langle \lambda+\rho,\alpha\rangle]:\textnormal{supp}(\alpha)=\textnormal{supp}(\lambda-\nu), B_\nu \in  \mathscr{W} \cdot A_{\alpha,r}\},$ and for each pair $(\alpha,r)\in I_\nu,$ choose $w_{\alpha,r}\in \mathscr{W}$ such that $w_{\alpha,r}\cdot A_{\alpha,r}=B_\nu.$ Then the contribution of $\nu$ to $\nu_\mu^c(T_\lambda)$ is given by 
$$
-\sum_{(\alpha,r)\in I_\nu}{\nu_p(r)\det (w_{\alpha,r})},
$$
where for $\ell$  a prime number  and $m\in \Z,$ we write $\nu_\ell(m)$ to denote the greatest integer $r$ such that $\ell^r$ divides $m.$
\end{thm}

\subsection{Computing certain weight multiplicities}     
\label{Computing_multiplicities}                         

We now use the results introduced in the previous section to investigate the structure of certain Weyl modules, as well as to compute various weight multiplicities in certain irreducible modules. For $\lambda\in X^+(T)$ and $c_1,\ldots,c_n\in \mathbb{Z}_{\geq 0},$ we use the notation $\lambda-c_1c_2\cdots c_n$ for the weight $\lambda-\sum_{r=1}^n{c_r\alpha_r}.$

\begin{lem}\label{C.f. of [a,b,a] for type B}
Let $G$ be of type $B_3$ over $K,$ and let $a,b\in \mathbb{Z}_{>0}$ be such  $a+b+1=p.$ Also let $\lambda=a\lambda_1+b\lambda_2+a\lambda_3\in X^+(T)$ and $\mu=\lambda-111\in \Lambda^+(\lambda).$ Then $\mu$ affords the highest weight of a composition factor of $V_G(\lambda)$ if and only if $(a,b)=(p-2,1).$
\end{lem}

\begin{proof} 
We first deal with the situation where  $p=3,$ $a=b=1,$ so that we have $p=a+2$ in this case. Here the only dominant $T$-weights $\nu\in X^+(T)$ such that $\mu \preccurlyeq \nu \prec \lambda$ are $\lambda-110,$ $\lambda-011,$ and $\mu$ itself. Now an application of Proposition \ref{Corollary to The linkage principle} shows that $\lambda-011$ cannot afford the highest weight of a composition factor of $V_G(\lambda).$ On the other hand $[V_G(\lambda),L_G(\lambda-110)]=1$ by Proposition \ref{8.6_Seitz}, but $\m_{L_G(\lambda-110)}(\mu)=0$ (as $\lambda-110 =3\lambda_3$) and hence
$$
\m_{L_G(\lambda)}(\mu)=\m_{V_G(\lambda)}(\mu)-[V_G(\lambda),L_G(\mu)].
$$
Now $\m_{L_G(\lambda)}(\mu)=3<4=\m_{V_G(\lambda)}(\mu)$ by \cite{Lueweb}, from which we deduce the existence of a second composition factor of $V_G(\lambda).$ Now since no weight $\nu$ with $\mu \prec \nu \prec \lambda$ affords the highest weight of a composition factor of $V_G(\lambda),$ the desired result follows in this case. 

We thus assume $p>3$ in the remainder of the proof and fix a Jantzen filtration $V_G(\lambda)=V^0\supsetneq V^1\supseteq \ldots \supseteq V^k\supseteq 0$ of $V_G(\lambda).$ We  start by computing all contributions to $\nu_\mu^c(T_\lambda).$ Here the dominant $T$-weights $\nu\in X^+(T)$ such that $\mu \preccurlyeq \nu \prec \lambda$ are $\lambda-100$ (if $a>1$), $\lambda-010$ (if $b>1$), $\lambda-001$ (if $a>1$), $\lambda-101$ (if $a>1),$ $\lambda-110,$ $\lambda-011,$ and $\mu$ itself. Now  $\lambda-100,$  $\lambda-010,$ $\lambda-001,$ and $\lambda-101$ all have multiplicity $1$ in $V_G(\lambda)$ and hence  none of them can afford the highest weight of a composition factor of $V_G(\lambda)$ by \cite{premet}. Recursively applying Lemma \ref{no_contribution_for_certain_weights}  then shows that those same weights  cannot contribute to $\nu^c_\mu(T_\lambda).$ In addition, an application of Proposition \ref{Corollary to The linkage principle} yields $[V_G(\lambda),L_G(\lambda-011)]=0,$ and hence $\lambda-011$ does not contribute to $\nu_\mu^c(T_\lambda)$ by Lemma \ref{no_contribution_for_certain_weights} again.

We now compute the contribution of $\nu_1=\lambda-110$ to $\nu^c_\mu(T_\lambda)$ by first determining all pairs $(\alpha,r)\in I_{\nu_1}$ as in Theorem \ref{How_to_determine_contributions}. A straightforward computation yields 
$$
B_{\nu_1}=\tfrac{1}{2}(3a+2b+3,a+2b+3,a+3),
$$
and since $\lambda-\nu_1$ has support $\{\alpha_1,\alpha_2\},$ we get that $\alpha=\alpha_1+\alpha_2=\varepsilon_1-\varepsilon_3$ by definition of $I_{\nu_1}.$ For $r\in \mathbb{Z},$ we have 
$$ 
A_{\varepsilon_1-\varepsilon_3,r}=\tfrac{1}{2}(3a+2b +5-2r, a+2b+3,a+1+2r).
$$
Recall from Bourbaki \cite[Planche II]{Bourbaki} that $\mathscr{W}$ acts by all permutations and sign changes  of the $\varepsilon_i,$ hence the orbit of any $3$-tuple $C=(c_1,c_2,c_3)\in \mathbb{Q}^3$ is given by 
$$
\mathscr{W}\cdot C=\{\pm c_{\sigma(1)},\pm c_{\sigma(2)}, \pm c_{\sigma(3)}:\sigma \in \mbox{Perm}(1,2,3)\}.
$$
We thus deduce that $A_{\varepsilon_1-\varepsilon_3,r}$ and $B_{\nu_1}$ are conjugate under the action of $\mathscr{W}$ if and only if $\{3a+2b+5-2r,a+1+2r\}=\{|3a+2b+3|,|a+3|\}.$ By studying each possibility separately, one easily shows that $A_{\varepsilon_1-\varepsilon_3,r}$ is $\mathscr{W}$-conjugate to $B_{\nu_1}$ if and only if $r=p,$ in which case the chosen element
$$
w_{\varepsilon_1-\varepsilon_3,r}=s_{\varepsilon_1-\varepsilon_3}
$$
satisfies $w_{\varepsilon_1-\varepsilon_3,r}\cdot  A_{\alpha,r}  =B_{\nu_1}$ as desired. Consequently an application of Theorem \ref{How_to_determine_contributions} shows that $\nu_1$ contributes to $\nu_\mu^c(T_\lambda)$ by $\nu_p(p)=1.$ We next determine the contribution of $\mu$ to $\nu^c_\mu(T_\lambda).$ Here again, a computation  yields
$$
B_{\mu}=\tfrac{1}{2}(3a+2b+3,a+2b+3,a+1),
$$
and since $\lambda-\mu$ has support $\Pi,$ we get that  $\alpha\in \{\varepsilon_1,\varepsilon_1+\varepsilon_2, \varepsilon_1+\varepsilon_3\}.$ Dealing with each possibility separately, one then concludes that $\mu$ contributes to $\nu^c_\mu(T_\lambda)$ by $\nu_p(3a+2b+4)$ by Theorem \ref{How_to_determine_contributions}, so 
$$
 \nu^c_\mu(T_\lambda)=  \chi_\mu(\lambda-110) + \nu_p(3a+2b+4)\chi_\mu(\mu).
 $$
Now  $\chi_\mu(\lambda-110)=\ch L_G(\lambda-110) + \delta_{p,a+2}\ch L_G(\mu)$ and $\chi_\mu(\mu)=\ch L_G(\mu).$ An application of Proposition \ref{Insight_on_possible_composition_factors} then completes the proof.
\end{proof}

\begin{lem}\label{C.f. of [a,b] for type B}
Let $G$ be  of type $B_n$ $(n\geq 2)$ over $K,$ and let $a,b\in \mathbb{Z}_{>0}$ be such that $b>1,$ and $p\mid (a+b+n-1).$ Also let  $\lambda=a\lambda_1+b\lambda_n\in X^+(T_G)$ and  $\mu=\lambda-1\ldots 12\in \Lambda^+(\lambda).$ Then $\mu$ affords the highest weight of a composition factor of $V_G(\lambda).$ Furthermore if $n=2,$ then 
$$
[V_G(\lambda),L_G(\mu)]=1.
$$
\end{lem}

\begin{proof}
Fix a Jantzen filtration $V_G(\lambda)=V^0\supsetneq V^1\supseteq \ldots \supseteq V^k\supseteq 0$ of $V_G(\lambda).$ We proceed as in the proof of Lemma \ref{C.f. of [a,b,a] for type B}, starting  by computing   all contributions to $\nu^c_\mu(T_\lambda)$ in the case where $n=2.$ Here the only dominant $T$-weights $\nu\in X^+(T)$ satisfying $\mu \preccurlyeq \nu \prec \lambda$ are $\lambda-10$ (if $a>1$), $\lambda-01,$ $\lambda-02$ (if $b>3$),  $\lambda-11,$ and $\mu$ itself. Now each of $\lambda-10,$ $\lambda-01,$ and $\lambda-02$ has multiplicity at most $1$ in $V_G(\lambda)$ and hence 
$$
[V_G(\lambda),L_G(\lambda-10)]=[V_G(\lambda),L_G(\lambda-01)]=[V_G(\lambda),L_G(\lambda-02)]=0
$$
by \cite{premet}. Applying Lemma \ref{no_contribution_for_certain_weights}  then shows that  none of these can contribute to $\nu^c_\mu(T_\lambda).$ Also $[V_G(\lambda),L_G(\lambda-11)]=0 $ by Proposition \ref{Corollary to The linkage principle}, hence again showing that $\lambda-11$ does not contribute to $\nu_\mu^c(T_\lambda)$ by Lemma \ref{no_contribution_for_certain_weights}. Finally, we compute the contribution of $\mu,$ by first determining all pairs $(\alpha,r)\in I_\mu$ as in Theorem \ref{How_to_determine_contributions}. A straightforward computation yields
$$
B_{\mu}=\tfrac{1}{2}(2a+b+1,b-1),
$$
and since $\lambda-\mu$ has support $\Pi,$ we get that $\alpha\in \{\varepsilon_1,\varepsilon_1+\varepsilon_2\} $ by definition of $I_\mu.$ We   claim that for every $r\in \mathbb{Z}_{\geq 0},$ the $2$-tuple  $A_{\varepsilon_1,r}$ is not conjugate to $B_\mu $ under the action of the Weyl group $\mathscr{W} $ of $G.$  Indeed, as in the proof of Lemma \ref{C.f. of [a,b,a] for type B}, recall from \cite[Planche II]{Bourbaki} that $\mathscr{W}$ acts by all permutations and sign changes of the $\varepsilon_i.$ Now $A_{\varepsilon_1,r}=\tfrac{1}{2}(2a+b+3-2r,b+1),$ and since none of the two coordinates of $B_\mu$ is equal to $\pm \tfrac{1}{2}(b+1),$ we immediately deduce that $A_{\varepsilon_1,r} \cap (\mathscr{W} \cdot B_\mu )=\emptyset$ for every $r\in \mathbb{Z}.$ Now considering the positive root $\varepsilon_1 + \varepsilon_2,$ we have
$$
A_{\varepsilon_1+\varepsilon_2,r}=\tfrac{1}{2}(2a+b+3-2r,b+1-2r),
$$
from which we deduce that $A_{\varepsilon_1+\varepsilon_2,r}$ and $B_\mu$ are  conjugate under the action of $\mathscr{W}$ if and only if $\{2a+b+3-2r,b+1-2r \}=\{|2a+b+1|,|b-1|\}.$ By studying each possibility separately, one  shows that $A_{\varepsilon_1+\varepsilon_2,r}$ is $\mathscr{W}$-conjugate to $B_\mu$ if and only if $r=p,$ in which case the  element
$$
w_{\varepsilon_1+\varepsilon_2,r}=s_{\varepsilon_1-\varepsilon_2}s_{\varepsilon_1}s_{\varepsilon_2}
$$
satisfies $w_{\varepsilon_1+\varepsilon_2,r}\cdot (A_{\varepsilon_1+\varepsilon_2,r})=B_\mu$ as desired. Therefore $
\nu^c_\mu(T_\lambda)= \chi_\mu(\mu)$ by Theorem \ref{How_to_determine_contributions}, and since $\chi_\mu(\mu)=\ch L_G(\mu),$  we get that $ \nu _\mu^c(T_\lambda)=  \ch L_G(\mu).$ An application of  Proposition \ref{Insight_on_possible_composition_factors} then completes the proof.

Next we assume $n>2,$ in which case the dominant $T$-weights $\nu\in X^+(T)$ satisfying $\mu \preccurlyeq \nu \prec \lambda$ are $\lambda-\alpha_1$ (if $a>1$), $\lambda-\alpha_n,$ $\lambda-2\alpha_n$ (if $b>3$), $\lambda-\alpha_{n-1}-2\alpha_n$ (if $b>2$), $\lambda-(\alpha_1+\cdots +\alpha_n),$ and $\mu$ itself. Now arguing as in the $n=2$ case, one shows that neither of $\lambda-\alpha_1,$ $\lambda-\alpha_n,$ $\lambda-2\alpha_n,$ $\lambda-\alpha_1-\alpha_n,$ $\lambda-\alpha_{n-1}-2\alpha_n,$ nor $\lambda-(\alpha_1+\cdots+\alpha_n)$ can contribute to $\nu^c_\mu(T_\lambda).$ Finally, we compute the contribution of $\mu,$ by first determining all pairs $(\alpha,r)\in I_\mu$ as in Theorem \ref{How_to_determine_contributions}. Again, a computation   yields
$$
B_{\mu}=\tfrac{1}{2}(2a+b+2n-3,b+2n-3,b+2n-5,\ldots,b+3,b-1),
$$
and since $\lambda-\mu$ has support $\Pi,$ we get that $\alpha\in \{\varepsilon_1,\varepsilon_1+\varepsilon_l:2\leq l\leq n\}.$ Now for $r\in \mathbb{Z}_{\geq 0} $ and $\alpha\in \Phi^+,$ we have 
$A_{\alpha,r}=\tfrac{1}{2}(2a+b+2n-1,b+2n-3,b+2n-5,\ldots,b+1)-r\alpha.$ Now by Bourbaki \cite[Planche II]{Bourbaki} again , $\mathscr{W}$ acts by all permutations and sign changes of the $\varepsilon_i,$ from which one immediately deduces that $A_{\alpha,r}$ cannot be conjugate to $B_\mu$ if $\alpha\neq \varepsilon_1+\varepsilon_n.$ Also, a straightforward computation yields
$$
A_{\varepsilon_1+\varepsilon_n,r} = \tfrac{1}{2}(2a+b+2n-1-2r,b+2n-3, b+2n-5,\ldots,b+3,b+1-2r),
$$
from which we deduce that $A_{\varepsilon_1+\varepsilon_n,r}$ and $B_\mu$ are  conjugate under the action of $\mathscr{W}$ if and only if $\{2a+b+3+2n-1-2r,b+1-2r\}=\{|2a+b+2n-3|,|b-1|\}.$ By studying each possibility separately, one shows that $A_{\varepsilon_1+\varepsilon_n,r}$ is $\mathscr{W}$-conjugate to $B_\mu$ if and only if $r=a+b+n-1,$ in which case the element
$$
w_{\varepsilon_1+\varepsilon_n,r}=s_{\varepsilon_1-\varepsilon_n}s_{\varepsilon_1}s_{\varepsilon_n}
$$
satisfies $w_{\varepsilon_1+\varepsilon_n,r}\cdot (A_{\varepsilon_1+\varepsilon_n,r})=B_\mu$ as desired. Therefore $
\nu^c_\mu(T_\lambda)= \nu_p(a+b+n-1)\chi_\mu(\mu)$ by Theorem \ref{How_to_determine_contributions}, and since $\chi_\mu(\mu)=\ch L_G(\mu),$  we get that $ \nu _\mu^c(T_\lambda)= \nu_p(a+b+n-1) \ch L_G(\mu).$ An application of  Proposition \ref{Insight_on_possible_composition_factors} then completes the proof.
\end{proof}

We will also require some knowledge about the structure of certain Weyl modules for a simple group of type $A_n$ over $K.$ However, due to the complexity of the description of fundamental dominant weights in terms of an orthonormal basis of a Euclidean space for such $G,$ it is more convenient  to work in a group of type $B_{n+1},$ and then deduce the desired result for $A_n.$

\begin{lem}\label{C.f. of [a,b,...,c] for type A}
Let $G$ be of type $B_{n+1}$ $(n\geq 3)$ over $K,$ and let $0<a,b,c<p$ be   positive integers. Also let   $\lambda=a\lambda_1+b\lambda_2+c\lambda_n\in X^+(T),$ $\nu_1=\lambda-\alpha_1-\alpha_2,$ $\nu_2=\lambda-(\alpha_2+\cdots + \alpha_n),$ and $\mu=\lambda-(\alpha_1+\cdots+\alpha_n).$ Then  $\mu$ affords the highest weight of a composition factor of $V_G(\lambda)$ if and only if $p\mid (a+b+c+n-1).$ Also if $n=3,$ then we have 
$$
\chi_\mu(\lambda) = \ch L_G(\lambda)  + \delta_{p,a+b+1} \ch L_G(\nu_1)	+ \delta_{p,b+c+1} \ch L_G(\nu_2)+  \delta_{p,a+b+c+2} \ch L_G(\mu).
$$
\end{lem}

\begin{proof} 
As in the proofs of Lemmas \ref{C.f. of [a,b,a] for type B} and \ref{C.f. of [a,b] for type B}, fix  $V_G(\lambda)=V^0\supsetneq V^1\supseteq \ldots \supseteq V^k\supseteq 0$ a Jantzen filtration for $V_G(\lambda).$ Arguing as in the aforementioned proofs, one first checks that 
$$
\nu^c_\mu(T_\lambda)=\nu_p(a+b+1)\chi_\mu(\nu_1) + \nu_p(b+c+1)\chi_\mu(\nu_2) + \nu_p(a+b+c+n-1)\chi_\mu(\mu).
$$
Now observe that $\chi_\mu(\nu_1)=\ch L_G(\nu_1) + \nu_p(c+n-2)\ch L_G(\mu),$ $\chi_\mu(\nu_2)=\ch L_G(\nu_2),$  and $\chi_\mu(\mu)=\ch L_G(\mu).$ Therefore an application of  Proposition \ref{Insight_on_possible_composition_factors}  shows that $\mu$ affords the highest weight of a composition factor if and only if $p\mid (a+b+c+n-1)$ as desired. Finally, if $n=3,$ then we get that $\nu_p(a+b+1)=\delta_{p,a+b+1},$ $\nu_p(b+c+1)=\delta_{p,b+c+1},$ and $\nu_p(a+b+c+2)=\delta_{p,a+b+c+2}.$ Also $\chi_\mu(\nu_1)=\ch L_G(\nu_1)$ in this case, so that an application of Proposition \ref{Insight_on_possible_composition_factors} completes the proof.
\end{proof}

\begin{prop}\label{Weight_multiplicities}
Let $G,$ $\lambda,$ and $\mu$  be as in Table \textnormal{\ref{Table_of_multiplicities}}, with $a,b,c\in \mathbb{Z}_{>0}$ such that $ a+b+1=p.$ Then the multiplicity   of $\mu$ in  $V_G(\lambda),$ respectively   $L_G(\lambda),$ is given in the fourth, resp. fifth, column of the table.
\begin{table}[h]
\center
\begin{tabular}{ccccc}
\hline\hline  \\[-1.5ex]
$G$			&	$\lambda$	& $\mu$	&	$\m_{V_G(\lambda)}(\mu)$ &	$\m_{L_G(\lambda)}(\mu)$\\[0.15cm] 
\hline\hline  \\[-1.5ex]
$A_3$		& $a\lambda_1+b\lambda_2+c\lambda_3$ & $\lambda-111$ &$ 4$   &$ 3-\delta_{a,c}$\\[0.15cm]  
\hline \\[-1.5ex]
$A_4$ 		& $a\lambda_1+b\lambda_2+a\lambda_3$  & $\lambda-1121$  & $6-2\delta_{a,1}$  &$ 3-\delta_{a,1}$   \\[0.15cm]  
\hline \\[-1.5ex]
$B_2$		& $  a\lambda_1+b\lambda_2$ & $\lambda-12$  & $ 3-\delta_{b,1}$  &$2$  \\[0.15cm] 
\hline\cr
\end{tabular}
\caption{Some weight multiplicities in various irreducibles. Here $a,b,c\in \mathbb{Z}_{>0}$ are such that $ a+b+1=p.$}
\label{Table_of_multiplicities}
\end{table}
\end{prop}

\begin{proof}
Let $G,$ $\lambda,$ and $\mu$ be as in the first row of Table \ref{Table_of_multiplicities}, and start by observing that an application of  \cite[Proposition 3]{Cavallin} yields  $\m_{V_G(\lambda)}(\mu)=4.$ Also, the weights $\lambda-100,$ $\lambda-010,$ $\lambda-001,$ and $\lambda-101$ all have multiplicity one in $V_G(\lambda)$ and hence none of them can afford the highest weight of a composition factor of $V_G(\lambda)$ by \cite{premet}. Setting $\mu_1=\lambda-110$ and $\mu_2=\lambda-011,$ we have $[V_G(\lambda),L_G(\mu_1)]=1$ by Proposition \ref{8.6_Seitz}, so that
\begin{equation}
\m_{L_G(\lambda)}(\mu)	=	\m_{V_G(\lambda)}(\mu)  - \m_{L_G(\mu_1)}(\mu) -[V_G(\lambda),L_G(\mu_2)]\m_{L_G(\mu_2)}(\mu) 	-[V_G(\lambda),L_G(\mu)].
\label{Equation_in_Proposition_on_multiplicities}
\end{equation}
Now if $a=c,$ then $\mu_2$ also affords the highest weight of exactly one composition factor of $V_G(\lambda)$ by Proposition \ref{8.6_Seitz}, and one easily sees that $\m_{L_G(\mu_1)}(\mu)=\m_{L_G(\mu_2)}(\mu)=1.$ Also, applying Proposition   \ref{Corollary to The linkage principle} yields $[V_G(\lambda),L_G(\mu)]=0,$ from which the desired result follows. If on the other hand $a\neq c,$ then $[V_G(\lambda),L_G(\mu_2)]=0,$ and we must consider separately the cases  $c=p-1$ or $c\neq p-1.$ In the former case, we get that $[V_G(\lambda),L_G(\mu)]=1$ by Lemma \ref{C.f. of [a,b,...,c] for type A}, while $\m_{L_G(\mu_1)}(\mu)=0,$ as $\mu_1=(a-1)\lambda_1+(b-1)\lambda_2+p\lambda_3$ in this case. The assertion then immediately follows from \eqref{Equation_in_Proposition_on_multiplicities}. Finally, if $c\neq p-1,$ then $\mu$ does not afford the highest weight of a composition factor of $V_G(\lambda)$ by Lemma \ref{C.f. of [a,b,...,c] for type A}, while $\m_{L_G(\mu_1)}(\mu)=1,$ as $\mu_1$ is $p$-restricted in this case. Again \eqref{Equation_in_Proposition_on_multiplicities} then yields the desired result. 

Next consider $G,$ $\lambda,$ $\mu$ as in the second row of Table \ref{Table_of_multiplicities} and observe that if $a=1,$ then the result is an immediate consequence of the previous case, as $\mu$ is conjugate to $\lambda-1110$ in this situation. So assume $a>1$ in the remainder of the argument. A recursive application of \cite[Proposition 1, Theorem 2]{Cavallin} yields $\m_{V_G(\lambda)}(\mu)=6.$ Also neither $\lambda-1110,$ $\lambda-0120,$ $\lambda-0121,$ $\lambda-1120,$ nor $\mu$ can afford the highest weight of a composition factor of $V_G(\lambda)$ by Proposition \ref{Corollary to The linkage principle}. The remaining potential highest weights of composition factors, except for $\mu_1=\lambda-1100$ and $\mu_2=\lambda-0110,$ all have multiplicity $1$ in $V_G(\lambda),$ so that   
$$
\m_{L_G(\lambda)}(\mu)=\m_{V_G(\lambda)}(\mu)-[V_G(\lambda),L_G(\mu_1)]\m_{L_G(\mu_1)}(\mu)-[V_G(\lambda),L_G(\mu_2)]\m_{L_G(\mu_2)}(\mu),
$$
by \cite{premet}. Finally, notice that $\m_{L_G(\mu_1)}(\mu)=1,$  while $\m_{L_G(\mu_2)}(\mu)=2$ by Proposition \ref{8.6_Seitz}, hence the result in this situation as well. 

Finally, let $G,$ $\lambda,$ $\mu$ be as in the third row of the table. If $b=1,$ then $\mu$ is conjugate to $\lambda-11,$ whose multiplicity in $V_G(\lambda)$ equals $2.$ An application of Proposition \ref{Corollary to The linkage principle} then yields the desired result in this case. If on the other hand $b>1,$ then $\mu$ is dominant and $\m_{V_G(\lambda)}(\mu)=3$ by \cite[Proposition 1, Theorem 2]{Cavallin}. By \cite{premet}, $[V_G(\lambda),L_G(\lambda-10)]=[V_G(\lambda),L_G(\lambda-01)]=0,$ while applying Proposition \ref{Corollary to The linkage principle}  shows that $\lambda-11$ does not afford the highest weight of a composition factor of $V_G(\lambda).$ Therefore $\m_{L_G(\lambda)}(\mu)=\m_{V_G(\lambda)}(\mu)-[V_G(\lambda),L_G(\mu)]$ and hence the assertion   follows from Lemma \ref{C.f. of [a,b] for type B}.  
\end{proof}

\subsection{Proof of Theorem \ref{F_4_in_E_6} and conclusion}
\label{3rd_result:conclusion}

Let $Y$ be a simple algebraic group of type $E_6$ over $K,$ and throughout this section, assume Hypothesis~\ref{inductive_hypothesis} for all embeddings $H\subset G$ with $\rank G < \rank Y.$ Fix $T_Y$ a maximal torus of $Y$ and let $B_Y$ be a Borel subgroup of $Y$ containing $T_Y.$ Let $\Pi(Y)=\{\alpha_1,\ldots,\alpha_6\}$ denote the corresponding base for the root system $\Phi(Y)=\Phi^+(Y) \sqcup \Phi^-(Y)$ of $Y,$ where $\Phi^+(Y)$ and $\Phi^-(Y)$ denote the set of positive and negative roots, respectively. Also write $\lambda_1,\ldots,\lambda_6$ for the associated fundamental dominant weights. Consider the subgroup $X$ of type $F_4$ defined by
$$X=\langle x_{\pm \beta_j}(c):1\leq j\leq 4,~c\in K \rangle,$$
where $x_{\pm \beta_1}(c)=x_{\pm \alpha_2}(c),$ $x_{\pm \beta_2}(c)=x_{\pm \alpha_4}(c),$ $x_{\pm \beta_3}(c)=x_{\pm \alpha_3}(c)x_{\pm \alpha_5}(c),$ and $x_{\pm \beta_4}(c)=x_{\pm \alpha_1}(c)x_{\pm \alpha_6}(c),$ for $c\in K.$ Let $T_X$ be the maximal torus of $X$ defined by the $  x_{\pm \beta_i}(c)$ (see \cite[Section 4.3]{Carter}), and let $B_X$ be the Borel subgroup of $X$ generated by the $x_{\pm \beta_i}(c)$ and $T_X,$ so that $\Pi(X)=\{\beta_1,\beta_2,\beta_3,\beta_4\}$ is a corresponding base for the root system $\Phi(X)$ of $X.$ (Here again $\Phi(X)=\Phi^+(X)\sqcup \Phi^-(X)$ in the obvious way.) We first recall an argument from \cite{test}.

\begin{lem}\label{thm2_red} Let $X$ and $Y$ be as above and $\lambda\in X^+(T_Y)$ a non-trivial $p$-restricted weight such that 
$L_Y(\lambda)|_X$ is irreducible. Then up to graph automorphism, one of the following holds.
\begin{enumerate}
\item[\rm (i)]$\lambda = a\lambda_2+a\lambda_3+b\lambda_4$ for some $a,b\in \Z_{>0},$ with  $a+b=p-1$.
\item [\rm (ii)]  $\lambda=(p-3)\lambda_1$, for $p>3.$
\item [\rm (iii)] $\lambda = \lambda_1+(p-2)\lambda_3$, for $p>2$.
\end{enumerate} 
\end{lem}

\begin{proof} First we note that Theorem~\ref{Main_result_1} applied to the Levi embedding $B_3\subset D_4$ implies that there exists $i\in\{1,3,4,5,6\}$ such that 
$\langle\lambda,\alpha_i\rangle\ne 0$. Let $L_X'$ be the $C_3$ Levi factor of $X$, which lies in $L_Y'$, an 
$A_5$-type Levi factor of $Y$. Then by Hypothesis~\ref{inductive_hypothesis} (iv), Theorem~\ref{Main_result_1}, and up to taking graph automorphisms,  
$\lambda\in\{a\lambda_1, (p-1)\lambda_3+x\lambda_2$ $(p\neq 2),$ $ c\lambda_1+b\lambda_3+x\lambda_2$ $(cb\ne 0,$ $b+c=p-1),$ $b\lambda_3+a\lambda_4+x\lambda_2$ $(ab\ne 0,$ $a+b=p-1)\}$. We now refer the reader to the proof of \cite[(5.4)]{test}, which establishes that if $\lambda=a\lambda_1,$ then $p>3,$ $a=p-3$, and $\lambda\ne(p-1)\lambda_3+x\lambda_2$. The same proof shows that if $\lambda= c\lambda_1+b\lambda_3+x\lambda_2$, with $cb\ne 0$ and $b+c=p-1$, then $p>2$, $x=0$ and $b=p-2$. 
Finally, in the last case, where $\lambda=b\lambda_3+a\lambda_4+x\lambda_2$, with $ab\ne 0$ and $a+b=p-1$, the same
proof shows that
$x\ne 0$, and then Theorem~\ref{Main_result_1} shows that (i) holds. \end{proof}

We are now ready to give a proof of Theorem \ref{F_4_in_E_6}, and to conclude by giving a proof of Theorem \ref{Main_result_3}.

\begin{proof}[\textbf{Proof of Theorem \ref{F_4_in_E_6}}]
Let $\lambda\in X^+(T_Y)$ be a non-zero, $p$-restricted, dominant weight for $T_Y,$ and write $V=L_Y(\lambda)$ for the 
corresponding irreducible $KY$-module. First observe that if $\lambda$ and $p$ are as in (i) or (ii)  of Theorem \ref{Main_result_3}, then $X$ acts irreducibly on 
$V=L_Y(\lambda)$ by \cite[5.6, 5.7]{test}.  
By Lemma~\ref{thm2_red}, in order to complete the proof, it only remains to show that $V|_X$ is reducible in the situation where $\lambda=a\lambda_2+a\lambda_3+b\lambda_4$ for some $a,b\in \mathbb{Z}_{>0}$ such that $a+b+1=p,$ which we shall assume holds in the remainder of the proof. Here the restriction of $\lambda$ to $T_X$ is given by  $\omega=a\omega_1+b\omega_2+a\omega_3 \in X^+(T_X).$ Consider the dominant $T_X$-weight 
$$
\omega'=\omega-\beta_1-\beta_2-2\beta_3-\beta_4\in X^+(T_X).
$$
Then $\lambda-\alpha_1-\alpha_2-2\alpha_3-\alpha_4,$ $\lambda-\alpha_1-\alpha_2-\alpha_3-\alpha_4-\alpha_5,$ and $\lambda-\alpha_2-\alpha_3-\alpha_4-\alpha_5-\alpha_6$ are all $T_Y$-weights of $V$ restricting to $\omega'.$ Also the latter two are $\mathscr{W}_Y$-conjugate to $\lambda-\alpha_2-\alpha_3-\alpha_4,$ yielding  
$$
\m_{V|_X}(\omega') \geq \m_V(\lambda-\alpha_1-\alpha_2-2\alpha_3-\alpha_4) + 2\m_V(\lambda-\alpha_2-\alpha_3-\alpha_4).
$$
By applying Proposition \ref{Weight_multiplicities} to the Levi subgroups of $Y$ corresponding to the simple roots 
$\alpha_1,\ldots,\alpha_4$, respectively  $\alpha_2,\alpha_3,\alpha_4,$ we   get
\begin{equation}
\m_{V|_X}(\omega') \geq 7-\delta_{a,1}.
\end{equation}

Next observe that recursively applying  \cite[Theorem 2]{Cavallin} yields  $
\m_{V_X(\omega)}(\omega')=8-2\delta_{a,1}.$ Also, the $T_X$-weight $\mu=\omega-\beta_1-\beta_2$ affords the highest weight of a composition factor of $V_X(\omega)$ by Proposition \ref{8.6_Seitz}, namely
$$
L_X(\mu)=(a-1)\omega_1+(b-1)\omega_2+(a+2)\omega_3.
$$
We now compute an upper bound for $\m_{L_X(\omega)}(\omega')$ by dealing with each of the following four possibilities separately.
 
\begin{enumerate}
\item If $a=1$ and $p \neq 3,$ then  $\mu$ is $p$-restricted and so $\m_{L_X(\mu)}(\omega')=1.$ Therefore we have  $\m_{L_X(\omega)}(\omega')\leq \m_{V_X(\omega)}(\omega')-1=5<6=\m_{V|_X}(\omega).$

\item If $a>1$ and $ a \neq p-2,$ then again $\mu$ is $p$-restricted and so $\m_{L_X(\mu)}(\omega')=1.$ Also by Lemma  \ref{C.f. of [a,b] for type B}, the weight $\nu=\omega  -\beta_2 - 2\beta_3$ affords the highest weight of a composition factor of $V_X(\omega),$ namely
$$
L_X(\nu)=(a+1)\omega_1 + b\omega_2+(a-2)\omega_3+2\omega_4.
$$
As $\m_{L_X(\nu)}(\omega')=1,$ we have $\m_{L_X(\omega)}(\omega')\leq 6<7=\m_{V|_X}(\omega').$
\item If $a=1$ and $p=3,$ then $\mu$ is not $p$-restricted and $\m_{L_X(\omega)}(\omega')=0.$ However, applying Lemma \ref{C.f. of [a,b,a] for type B} shows that  $\omega-\beta_1-\beta_2-\beta_3=\omega_2+\omega_3+\omega_4$ affords the highest weight of a composition factor of $V_X(\omega).$ Consequently, as $\m_{L_X(\omega_2+\omega_3+\omega_4)}(\omega')=1,$ we get that $\m_{L_X(\omega)}(\omega')\leq 5<6=\m_{V|_X}(\omega').$
\item If $a>1$ and $ a =p-2,$ then again $\omega'$ is not a weight of $L_X(\mu).$ However, an application of Lemma  \ref{C.f. of [a,b,a] for type B} (resp. Lemma \ref{C.f. of [a,b] for type B}) to the Levi subgroup of $X$ corresponding to the simple roots $\beta_1,\beta_2,\beta_3$ (resp. $\beta_2,\beta_3$) yields the existence of a composition factor of $V_X(\omega)$ having highest weight $\omega-\beta_1-\beta_2-\beta_3$ (resp. $\omega-\beta_2-2\beta_3$). Therefore, as $\mu$ is a weight of each of these irreducible, we get that $\m_{L_X(\omega)}(\omega')\leq 6<7=\m_{V|_X}(\omega').$
\end{enumerate}
Consequently, in each case, we get that $\m_{V|_X}(\omega')>\m_{L_X(\omega)}(\omega'),$ showing the existence of a second composition factor of $V$ for $X.$ In particular $V|_X$ is reducible. This completes the proof of Theorem \ref{F_4_in_E_6}.
\end{proof}

\begin{proof}[\textbf{Proof of Theorem \ref{Main_result_3}}]
Let $X \subset Y=E_n$ be as in the statement of Theorem~\ref{Main_result_3}. We embed $X$ in a maximal 
proper closed connected subgroup of $Y$ and deduce by the remarks at the beginning of Section~\ref{exceptional}, Proposition~\ref{other_embeddings}, and 
Theorem~\ref{F_4_in_E_6},
that $X$ lies in the $F_4$ subgroup of $Y=E_6$. So we must determine the irreducible configurations 
$X \subset F_4$ acting irreducibly on the irreducible $F_4$-module with highest weight $\lambda\in\{(p-3)\lambda_1 \ 
(p>3), \lambda_1+(p-2)\lambda_3\  (p>2)\}$; this is covered by \cite[Main Theorem]{test}. Then \cite[Main Theorem (iii), (iv)]{test}
 completes the proof of Theorem~\ref{Main_result_3}.
\end{proof}

\end{document}